\documentclass[11pt,a4paper]{amsart}
\usepackage[foot]{amsaddr}
\usepackage{texmac}
\usepackage{enumerate}

\usepackage[color,notcite,notref]{showkeys}   
\definecolor{labelkey}{rgb}{0.5,0.5,1}

\operators{Stab,Prim,Inn,Out}
\let\prim\Prim
\bbsymbols{b}{C,F,Z,Q}
\calsymbols{c}{H,U,X,Y,P,M}
\let\le\leqslant
\let\ge\geqslant

\def\qe#1{$#1$-quasi-elementary}
\def\dsum_#1_#2{\sum_{{#1}\atop {#2}}}
\begin{document}

\title{Brauer relations in finite groups}
\author{Alex Bartel$^1$}
\address{$^1$Mathematics Institute, Zeeman Building, University of Warwick,
Coventry CV4 7AL, UK}
\author{Tim Dokchitser$^2$}
\address{$^2$Department of Mathematics, University Walk, Bristol BS8 1TW, UK}
\email{a.bartel@warwick.ac.uk, tim.dokchitser@bristol.ac.uk}
\llap{.\hskip 10cm} \vskip -0.8cm
\ufootnote{To appear in J. Eur. Math. Soc.
  \\[-7pt]}
\maketitle

\begin{abstract}
If $G$ is a non-cyclic finite group, non-isomorphic $G$-sets $X, Y$ may
give rise to isomorphic permutation representations $\C[X]\iso \C[Y]$.
Equivalently, the map from the Burnside ring to the rational representation ring
of $G$ has a kernel. Its elements are called Brauer relations,
and the purpose
of this paper is to classify them in all finite groups,
extending the Tornehave-Bouc classification in the case of
$p$-groups.
\end{abstract}

\tableofcontents

\newpage

\section{Introduction}\label{s:intro}

\subsection{Background and main result}
The Burnside ring $B(G)$ of a finite group $G$ is the free abelian group
on isomorphism classes of finite $G$-sets modulo the relations
$[X]+[Y]=[X\amalg Y]$ and with multiplication $[X]\cdot[Y]=[X\times Y]$.
There is a natural ring homomorphism from the Burnside ring to the rational
representation ring of $G$,
$$
   B(G) \lar R_\Q(G), \qquad X\mapsto \Q[X].
$$
The purpose of this paper is to describe its kernel $K(G)$.

Both the kernel and the cokernel have been studied extensively.
The cokernel is finite of exponent dividing $|G|$ by Artin's induction
theorem, and Serre  remarked that it need not be
trivial (\cite{SerLi} Exc. 13.4).
It is trivial for $p$-groups \cite{Feit,Ritter,Segal}
and it has been determined in many special cases \cite{DK,Berz,HT}.

Elements of the kernel $K(G)$ are called
\emph{Brauer relations} or \emph{($G$-)relations}.
The most general result on $K(G)$ is due to
Tornehave \cite{Tor-84} (see \cite[2.4]{Kah-08}) and
Bouc \cite{Bouc}, who independently described it for $p$-groups.

There is a bijection $H\mapsto G/H$ between conjugacy
classes of subgroups of~$G$ and isomorphism classes of
transitive $G$-sets, and we will write elements $\Theta\in B(G)$ as
$\Theta=\sum_i n_i H_i$ using this identification.
In this notation,
$$
  \Theta \in K(G) \quad\Longleftrightarrow\quad
  \sum_i n_i\Ind_{H_i}^G\triv_{H_i}=0.
$$
If we allow inductions of arbitrary 1-dimensional characters instead of
just the trivial character,
isomorphisms between sums of such inductions are called monomial
relations. Deligne \cite[\S 1]{Del-73} described all monomial relations in
soluble groups, following Langlands \cite{Lan-70}.
For arbitrary finite groups, a generating set of monomial relations
was given by Snaith \cite{Sna-88}.

Following the approach of Langlands, Deligne, Tornehave and Bouc,
we consider a relation ``uninteresting'' if it is induced from a proper subgroup
or lifted from a
proper quotient of $G$ (see \S\ref{s:first}).
We call a relation \emph{imprimitive} if
it is a linear combination of such relations from proper subquotients
and \emph{primitive} otherwise, and we let $\Prim(G)$ denote
the quotient of $K(G)$ by the subgroup of imprimitive relations. The motivation
for this approach is that if one wants to prove a statement that holds for
all Brauer relations, and if this statement behaves well under induction
and inflation, then it is enough to prove it for primitive relations
(see also \S \ref{sec:applics}).

In this paper we classify finite groups that have primitive relations
and determine $\Prim(G)$:

\newpage

\namedthm{A}
\label{thm:A}
Let $p$ and $l$ denote prime numbers.
A finite non-cyclic group $G$ has a primitive relation if and only if either
\leftmargini1.7em
\begin{enumerate}
\item $G$ is dihedral of order $2^n\ge 8$; or
\item $G=(C_p\times C_p)\rtimes C_p$ is the Heisenberg group of order $p^3$
      with $p\ge 3$;~or
\item $G$ is an extension
$$
  1 \lar S^d \lar G \lar Q \lar 1,
$$
where $S$ is simple, $Q$ is quasi-elementary, the natural map
\hbox{$Q\!\rightarrow\!\Out S^d$} is injective and, moreover, either
\begin{enumerate}
\item $S^d$ is minimal among the normal subgroups of $G$\\
      (for soluble $G$, this is equivalent to $G\cong \bF_l^d\rtimes Q$ with
      $\bF_l^d$ a faithful irreducible representation of
      $Q$) or
\item $G=(C_l\rtimes P_1)\times (C_l\rtimes P_2)$
      with cyclic (possibly trivial) $p$-groups $P_i$ that act faithfully
      on $C_l\times C_l$ with $l\ne p$; or
\end{enumerate}
\item $G=C\rtimes P$ is quasi-elementary, $P$ is a $p$-group,
      $|C|=l_1\cdots l_t>1$ with $l_i\neq p$ distinct primes,
      the kernel $K=\ker(P\rightarrow \Aut C)$ is trivial,
      or isomorphic to $D_8$, or has normal
      $p$-rank one (see Proposition \ref{prop:prankone}). Moreover,
      writing $K_j=\bigcap_{i\neq j} \ker(P\rightarrow \Aut C_{l_i})$,
      either
\begin{enumerate}
\item $K=\{1\}$, $t>1$, and all $K_j$ have the same non-trivial image in the
Frattini quotient of $P$; or
\item $K\cong C_p$, $P\iso K\times (P/K)$, and all $K_j$ have the same
two-dimensional image in the Frattini quotient of $P$; or
%
\item $|K|>p$ or $P$ is not a direct product by $K$,
      and the graph $\Gamma$ attached to $G$ by
      Theorem \ref{thm:Knontriv} is disconnected.
\end{enumerate}
\end{enumerate}
\endnamedthm

\noindent
For these groups, $\Prim(G)$ is as follows. We write $\mu$ for the M\"obius
function.

$$
\begingroup
\def\baselinestretch{1}
\def\l#1{\langle#1\rangle}
\def\vrt#1#2{\begin{array}{c}\hskip -0.45em#1\cr\hskip -0.45em#2\end{array}}
\def\vrtt#1#2#3{\begin{array}{c}\hskip -0.45em#1\cr\hskip -0.45em#2\cr\hskip -0.45em#3\end{array}}
\def\vrttt#1#2#3#4{\begin{array}{c}\hskip -0.45em#1\cr\hskip -0.45em#2\cr\hskip -0.45em#3\cr\hskip -0.45em#4\end{array}}
\def\vrtttt#1#2#3#4#5{\begin{array}{c}\hskip -0.45em#1\cr\hskip -0.45em#2\cr\hskip -0.45em#3\cr\hskip -0.45em#4\cr\hskip -0.45em#5\end{array}}
\def\tif{\text{ if }}
\def\telse{\text{ else}}
\def\vc#1#2{\smash{\raise#1pt\hbox{$#2$}}}
\def\vp#1{\vphantom{\vbox to #1pt{\vfill}}}
\hskip -1cm
\begin{array}{|@{\>}c@{}|@{\>}c@{}|l@{\hskip-1pt}|}
\hline
\text{$\!$Case$\!$ }&\Prim(G)&\text{Basis of $\Prim(G)$}\cr
\hline
\text{1}\vp{17}  &  \Z/2\Z  &
\vrt{\Theta = H\!-H'\!+\!ZH'\!-\!ZH,\hfill}{\quad\>\>
\text{\smaller[3]$H\iso C_2$ and $H'\iso C_2$ are non-conjugate non-central,
$Z=Z(G)\iso C_2$}}\cr
\hline
\text{2} & (\Z/p\Z)^{p} &
  \vrt{\Theta_j = \l{y}\!-\!\l{xy^j}\!-\!\l{y,z}\!+\!\l{xy^j,z},
    \>\>1\!\le\!j\!\le\! p,}{\quad\>\>
    \text{\smaller[3]{$G=\l{x,z}\!\rtimes\!\l{y},\>\> z\in Z(G)$}}\hfill}
   \cr
\hline
\vbox to 4.5em{\vskip 6.5mm\hbox{\hfill $\,$3a\hfill}\vskip 6.5mm
      \hbox to 1.5em{\hrulefill$\>$}\vskip 5.5mm\hbox{\hfill $\,$3b\hfill}\vskip 3.5mm}
&
\vbox to 4.5em{\vskip 1mm\hbox{\hfill
      $\vrtt{\Z\tif Q\text{ cyclic}}{\vphantom{\int^X}\Z/p\Z\telse}
            {\raise4pt\hbox{\text{\smaller[4]($Q$ $p$-quasi-elementar\rlap{y)}}}}$
      \hfill}\vskip 0mm
      \hbox to 7em{$\>$\hrulefill}\vskip 1mm\hbox to 7em{\hfill$\>$
      $\vrtt{\Z\phantom{/p\Z}\tif Q\!=\!\{1\}}
           {\Z/p\Z\tif Q\!\ne\!\{1\}}
           {\raise2pt\hbox{\text{\smaller[4]($Q$=$P_1\times P_2$)}}}$
      \hfill}}
&
     \vrt{
     \left.\!\vrtttt{\Theta =
      C_{p^k}\!-\!p\,Q\!-\!C_l\!\rtimes\!C_{p^k} \!+\! p\,G\>\>
     \tif d\!=\!1, Q\!=\!C_{p^{k+1}}\hfill}
     {\Theta=G\!-\!Q\!+\!\alpha(C_n\!-\!C_l\!\rtimes\! C_n)\!+\!\beta(C_m\!-\!C_l\!\rtimes\! C_m)\hfill}
         {\quad\>\>\text{\smaller[2]if $d=1, Q=C_{mn}, \alpha m+\beta n=1$ (any such $m,n>1$)}\hfill}
         {\Theta =G\!-\!Q\!+\!\sum_U(U\!\rtimes\! N_QU\!-\!(C_l)^d\!\rtimes\! N_QU)\hfill}
         {\quad\>\>\text{\smaller[2]if $d\!>\!1$; sum over $U\!\subset\!(C_l)^d$ of index $l$ up to $G$-conjugacy}\hfill}
     \!\!\right\}
     \vrt{G\!\iso\!(C_l)^d\!\rtimes\! Q}{\text{soluble}}\hfill}
     {\Theta=\text{any relation of the form $G\!+\!\sum_{\scriptscriptstyle H\ne G} a_H H\hskip 1.9em$\raise2pt\hbox{$\bigr\}\>S\not\iso C_l$}}\hfill}
     \hfill\cr
\hline

\text{4a} &  \Z/p\Z
 & \vrt{\Theta = \sum_{U\le C}\mu(|U|)(MU\!-\!M'U),\> M,M'\!\le\!P\text{ of index $p$}\hfill}
    {\quad\>\>\text{\smaller[2]with signatures $(1,\ldots,1),(0,\ldots,0)$,
respectively (c.f. Prop. \ref{prop:condsKtriv}) }}\cr
\hline
\text{4b} & (\Z/p\Z)^{p-2} & \vrtt{\Theta_i =\sum_{U\le CK}\mu(|U|)(H_1 U\!-\!H_i U)\text{ for }1<i<p,\hfill}
{\quad\>\>\text{\smaller[3]$H_j\le P$ of index $p$, $K\nleq H_j$ and
$(H_j\cap K_1)\Phi(P)/\Phi(P)=L_j$,}\hfill}
{\quad\>\>\text{\smaller[3]%
$L_j$ distinct lines in $K_1\Phi(P)/\Phi(P)$ other than $K\Phi(P)/\Phi(P)$, $1\le j<p$}\hfill}\cr
\hline
\vc2{\text{4c}} &
  \vp{24}
  \vc0{\vrtt{(\Z/p\Z)^{d-1}}{\text{\smaller[3]$d=\#\>$connected}}{\raise4pt\hbox{\text{\smaller[3]components of $\Gamma$}}}}
&
\vrt{\Theta_i =\sum_{U\le C\centralCp}\mu(|U|)(H_1U\!-\!H_i U),\>\> 2\le i\le d,\hfill}
  {\quad\>\>\text{\smaller[3]$H_j$ any vertex in the $j$th connected component of $\Gamma$ and
  $\centralCp\iso C_p\le Z(K)$}\hfill}\\[9pt]
\hline
\end{array}
\endgroup
$$
\smallskip

\subsection{Overview of the proof}

Our analysis of finite groups follows a standard pattern

\begin{center}
abelian --- $p$-groups --- quasi-elementary --- soluble --- all finite,
\end{center}
\noindent
with a somewhat surprising twist that the difficulty of understanding
primitive relations seems to decrease from the middle to the sides.

It is classical that the only abelian groups that have primitive relations
are $G=C_p\times C_p$. On the opposite side,
Solomon's induction theorem together with the fact that
imprimitive relations form an ideal in
the Burnside ring immediately allows us to deal with a large class of groups:
if $G$ has a proper non-quasi-elementary quotient,
then $G$ has no primitive relations
(Corollary \ref{cor:alaDeligne} and Theorem
\ref{thm:notquasiel}(3)). Similarly, using Theorem \ref{thm:Tims},
we get the same conclusion when $G$ has non-cyclic quasi-elementary
quotients for two distinct primes $p\ne q$
(Theorem \ref{thm:notquasiel}), and deduce Theorem \ref{thm:A}
in the non-soluble case.
This strategy was inspired by Deligne's work on monomial relations.

The $p$-group case and the soluble case are somewhat more involved.
Our main tool for showing imprimitivity is the fact that in quasi-elementary
groups,
a relation $\sum n_H H$ with all $H$ contained in a proper subgroup of $G$
is imprimitive
(Proposition \ref{prop:biggroup}). This is surprisingly useful.
For instance, together with Bouc's `moving lemma' (\cite{Bouc} Lemma~6.15)
it gives an alternative proof of the Tornehave-Bouc classification
in the $p$-group case (see~\S\ref{sec:pgroups}).
The classification of primitive relations in soluble groups that are not
quasi-elementary is also not hard (see~\S\ref{sec:solvable}).

The most subtle case is that of quasi-elementary groups (\S\ref{sec:qegen}).
Recall that a quasi-elementary group is one of the form
$G=C\rtimes P$ with $P$ a \hbox{$p$-group} and $C$ cyclic of order coprime
to $p$. Assuming that such a $G$ has a primitive relation, we
analyse the kernel of the action of $P$ on $C$ (\S\ref{sec:kernel}) and
decompose all permutation representations of $G$ explicitly
into irreducible characters (\S\ref{sec:somerels}).
We show that $\Prim(G)$ is generated by relations of the form
$$
  \Theta = \sum_{U\le C\cdot Z(G)} \mu(|U|)(UH_1-UH_2),
$$
where $H_1,H_2\!\le\! G$ are of maximal size among those subgroups that intersect
$C\cdot Z(G)$ trivially, unless $Z(G)$ is trivial, in which case $H_1,H_2$
are of index $p$ in $P$.
This already settles Theorem \ref{thm:B} below, but the remaining issue
of primitivity of these generating relations is quite tricky.
To show that $\Theta$ as above is imprimitive, it is not enough
to show that it is neither lifted from a quotient nor induced from a
subgroup, since $\Theta$ could be a sum of relations each of which is either
lifted or induced.
It becomes necessary to explicitly split the maximal size subgroups
into classes in such a way that any relation involving two subgroups
from different classes has to be primitive. This is the
general spirit of sections \ref{sec:Ktriv} and \ref{sec:Knontriv},
which complete the proof of Theorem~\ref{thm:A}.

\smallskip

\subsection{Remarks and applications}\label{sec:applics}
Note that for non-soluble groups in Theorem \ref{thm:A}(3a),
$\Prim(G)$ is generated by any relation $\Theta=\sum_H n_H H$
with $n_G=\pm 1$ (Theorem \ref{thm:notquasiel}). An explicit construction
of such a relation can be found in \cite[Theorem 2.16(i)]{Sna-88}.
We note also that the relations in Theorem \ref{thm:A} for soluble groups are
fairly canonical, see e.g. Remark \ref{remcanrel}.

One of the reasons one is interested in Brauer relations comes from number
theory. In fact, the motivation for the Langlands-Deligne classification of 
monomial relations in soluble groups \cite{Lan-70,Del-73} was to build
a well-defined theory of $\epsilon$-factors of Galois
representations starting with one-dimensional characters;
to do this, one needs to prove that the $\epsilon$-factors
of one-dimensional characters cancel in all monomial relations
of local Galois groups.

If $F/\Q$ is a Galois extension of number fields,
arithmetic invariants of subfields $K\subset F$
may be viewed, via the Galois correspondence
$K\leftrightarrow\Gal(F/K)$, as functions of subgroups of $G=\Gal(F/\Q)$.
Some functions, such as the discriminant $K\mapsto \Delta(K)$ extended
to $B(G)\to\Q^\times$ by linearity, factor through the representation ring
$R_\Q(G)$ and so cancel in all Brauer relations.
On the other hand, the class number $h(K)$, the regulator $R(K)$
or the number of roots of unity $w(K)$ are not `representation-theoretic',
and do not cancel in general.
However, their combination $hR/w$ does, as it is
the leading term of the Dedekind $\zeta$-function $\zeta_K(s)$ at $s=1$,
and $\zeta$-functions are representation-theoretic by Artin formalism
for $L$-functions.

Thus, Brauer relations can provide non-trivial relationships between
different arithmetic invariants, like the class numbers and the regulators
of various intermediate fields. This point of view proved to be very
fruitful to study class numbers and unit groups \cite{Bra, Kur, Wal, Smi-01},
related Galois module structures \cite{BB-04,Bar-11}
and Mordell-Weil groups and other arithmetic invariants of
elliptic curves and abelian varieties \cite{squarity, tamroot, Bar-09}.
In a slightly different direction, a
verification of the vanishing in Brauer relations of \emph{conjectural}
special values of $L$-functions can be regarded as strong
evidence for the corresponding conjectures.
This has been carried out in the case of
the Birch and Swinnerton-Dyer conjecture in \cite{squarity} and in the
case of the Bloch-Kato conjecture in \cite{Burns}.


One concrete number-theoretic application of Brauer relations is the theory 
of `regulator constants', used in the proof of the
Selmer parity conjecture for elliptic curves over $\Q$ \cite{squarity}, 
questions related to Selmer growth \cite{tamroot, kurast, Bar-09}, and also 
to analyse unit groups and higher $K$-groups of number fields 
\cite{Bar-11, BarSm}. 
The regulator constant $\cC_\Theta(\Gamma)\in \Q^\times$ is an invariant
attached to a $\Z[G]$-module~$\Gamma$ and a Brauer relation $\Theta$ 
in $G$. For applications to elliptic curves the most important regulator
constant is that of the trivial $\Z[G]$-module $\Gamma=\triv$, as it
controls the $l$-Selmer rank of the curve over the ground field. 
For $\Theta=\sum_H n_H H$ it is simply
$$
  \cC_\Theta(\triv) = \prod_H |H|^{-n_H}.
$$
To deduce something about the Selmer rank, one relies on
Brauer relations in which this invariant, or rather its $l$-part,
is non-trivial. 
As an application of Theorem \ref{thm:A}, in \S\ref{s:appl} we settle
a question left unanswered in \cite{squarity, tamroot, kurast, Bar-09}, namely
which groups have such a Brauer relation. This is done in
Theorem \ref{thm:regconsttriv} and Corollary \ref{cor:regconstsubquo};
for an example of number theoretic consequences of this result, 
see~\cite{Bar-12}.

For such applications one needs a collection of Brauer relations
that span $K(G)$ and that are `as simple as possible', but
whether they are imprimitive is less important.
Theorem~\ref{thm:A} describes the {\em smallest} list of groups such~that 
all Brauer relations in all finite groups
come\footnote{We would like to propose to use the word {\em indufted}
  instead of a vague `come' or a cumbersome `induced and/or lifted',
  but we were not brave enough to do this throughout the paper}
from such subquotients.
Let us 
give an alternative version of the classification
theorem with a much cleaner 
set of generating relations, that avoids
the fiddly combinatorial conditions of
Theorem \ref{thm:A} (especially 4a,4b,4c).
It is a direct consequence of Theorem~\ref{thm:A}.

\namedthm{B}
\label{thm:B}
All Brauer relations in soluble groups are generated by relations
$\Theta$ from subquotients $G$ of the following three types.
In every case, $G$ is an extension $1\to C\to G\to Q\to 1$ with
$Q$ quasi-elementary and acting faithfully on $C$.
\begin{enumerate}
\item $C=C_l$, $l$ a prime (so $G=C\rtimes Q$), $H\le Q$,
and
$$
  \Theta=[Q\!:\!H]\>G-[Q\!:\!H]\>Q+H-CH.
$$
\item $C=(C_l)^d$, with $l$ a prime, $d\ge 2$, $G=C\rtimes Q$, and
$$
  \Theta = G-Q+\sum_U(U\rtimes N_QU-(C_l)^d\rtimes N_QU),
$$
the sum taken over representatives of $G$-conjugacy
classes of subgroups $U\le (C_l)^d$  of index $l$.
\item $C$ is cyclic, $Q$ is an abelian $p$-group,
$H_1, H_2\le G$ intersect $C$ trivially, $|H_1|=|H_2|$, and
$$
  \Theta = \sum_{U\le C} \mu(|U|)(UH_1-UH_2).
$$
\end{enumerate}
Conversely, all $\Theta\in B(G)$ of the listed type are Brauer relations,
not necessarily primitive.
Finally, relations from subquotients of type (1), (3) and
\begin{itemize}
\item[\textit{(2$^{\>\prime}$)}]
$C=S^d$ with $d\ge 1$ and $S$ simple, $G$ is not quasi-elementary, and
$$
  \Theta=\text{any relation of the form $G\!+\!\sum_{\scriptscriptstyle H\ne G} a_H H$}
$$
\end{itemize}
generate all Brauer relations in all finite groups.

\endnamedthm

\begin{acknowledgements}
We are really fortunate to have had several extremely conscientious
referees who made numerous corrections, pertinent remarks, and suggestions for
improvement, and who we are very grateful to.
We especially wish to thank the JEMS referee for 
an incredibly thorough reading of the manuscript. 

The second author is supported by a Royal Society
University Research Fellowship, and the first author was
partially supported by the EPSRC and is partially supported by
a Research Fellowship from the Royal Commission for the Exhibition
of 1851. Parts of this research were done at 
St Johns College, Robinson College and DPMMS in Cambridge, 
CRM in Barcelona, and Postech University in Pohang. 
We would like to thank these institutions
for their hospitality and financial support. 
\end{acknowledgements}

\subsection{Notation}
Throughout the paper, $G$ is a finite group; $Z(G)$ stands for the centre
of $G$ and $\Phi(G)$ for the Frattini subgroup;
whenever $Z(G)$ is a cyclic $p$-group, we write $\centralCp$ for the central
subgroup of order $p$;
we denote by $\triv$ the
trivial representation;
restriction from $G$ to $H$ and induction from $H$ to $G$ are denoted by
$\Res_H^G\rho$ and $\Ind_H^G\sigma$, respectively;
${}^gH$ stands for $gHg^{-1}$.

\section{First properties}\label{s:first}

Relations can be induced from and restricted to subgroups, and lifted
from and projected to quotients as follows:
let $\Theta=\sum_i n_i H_i$ be a $G$-relation.
\begin{itemize}
\item \textbf{Induction.} If $G'$ is a group containing $G$, then, by
transitivity of induction,
$\Theta$ can be induced to a $G'$-relation $\Ind^{G'}\Theta=\sum_i n_i H_i$.
\item \textbf{Inflation.} If $G\cong \tilde{G}/N$, then each $H_i$ corresponds
to a subgroup $\tilde{H_i}$
of $\tilde{G}$ containing $N$, and, inflating the permutation representations
from a quotient, we see that $\tilde{\Theta}=\sum_i n_i\tilde{H_i}$
is a $\tilde{G}$-relation.
\item \textbf{Restriction.} If $H$ is a subgroup of $G$, then by Mackey
decomposition $\Theta$ can be restricted to an $H$-relation
$$
  \Res_H\Theta=
  {\sum_i \Bigl(n_i\!\!\! \sum_{\scriptscriptstyle g\in H_i\backslash G/H}\!\!\!
H\cap {}^gH_i\Bigr)}.
$$
On the level of (virtual) $G$-sets this is simply the restriction of a $G$-set to $H$.
\item \textbf{Projection} (or \textbf{deflation}). If $N\normal G$, then
$N\Theta = \sum_i NH_i$ is a $G/N$-relation.
\end{itemize}
\begin{remark}
Note that by definition of multiplication in the Burnside ring,
$\Theta\cdot H = \Ind^G(\Res_H\Theta)$ for any $G$-relation
$\Theta$ and any subgroup $H\le G$.
\end{remark}

The number of isomorphism classes of irreducible rational representations of a
finite group $G$ is equal to the number of conjugacy classes of cyclic subgroups
of $G$ (see \cite[\S 13.1, Cor. 1]{SerLi}). Since the cokernel of
$B(G)\to R_{\Q}(G)$ is finite (see \cite[\S 13.1, Theorem 30]{SerLi}), the rank
of the kernel $K(G)$ is the number of conjugacy classes of non-cyclic subgroups.

Explicitly, Artin's induction theorem gives a relation for each non-cyclic
subgroup $H$ of $G$,
$$
  |H|\cdot\triv = \sum_C n_C C, \qquad n_C\in\Z,
$$
the sum taken over the cyclic subgroups of $H$. These are clearly linearly
independent, and thus give a basis of $K(G)\tensor\Q$.

\begin{example}
Cyclic groups have no non-zero relations.
\end{example}

\begin{example}\label{exer:G20}
Let $G=C_l\rtimes H$, with $l$ prime and $H\ne\{1\}$ acting
faithfully on $C$ (so $H$ is cyclic of order dividing $l-1$).
Let $\tilde{H}$ be any subgroup of $H$, set $\tilde{G} = C_l\rtimes \tilde{H}$.
Then,
$$
\tilde{H} - [H:\tilde{H}]\cdot H - \tilde{G} + [H:\tilde{H}]\cdot G
$$
is a relation. This can be checked by a direct computation,
using the explicit description of irreducible characters of $G$ in
Remark \ref{rmrk:charSemiDir}. (See e.g. Corollary \ref{cor:relmodC}.)
\end{example}

\begin{example}\label{ex:cpcp}

Let $G=C_p\times C_p$. All its proper subgroups are cyclic, so $K(G)$ has rank one.
It is generated
by $\Theta = 1 - \sum_C C + pG$, with the sum running over all subgroups of order $p$, 
as can be checked by an explicit decomposition into
irreducible characters, as above
(or see \cite{Bouc} or Proposition \ref{prop:SerreGroups} below).
\end{example}

\section{Imprimitivity criteria} \label{s:imprim}

\begin{lemma}\label{lem:deflations}
Let $G$ be a finite group, and
$\Theta = \sum_i n_i H_i$ a $G$-relation
in which each $H_i$ contains some non-trivial normal subgroup $N_i$ of $G$.
Then $\Theta$ is imprimitive.
\end{lemma}
\begin{proof}
Subtracting the projection onto $N_1$, we get a relation
$$
  \Theta - N_1\Theta = \sum_{i, H_i \ngeq N_1} n_i(H_i - N_1H_i),
$$
which consists of subgroups each of which contains one of $N_2,\ldots,N_k$.
Repeatedly replacing $\Theta$ by $\Theta-N_j\Theta$ we see that the remaining
relation is zero,
so $\Theta$ is a sum of relations that are lifted from quotients.
\end{proof}

\begin{lemma}\label{lem:noncycliccentre}
Let $G\not\iso C_p\times C_p$ be a finite group with non-cyclic centre.
Then $G$ has no primitive relations.
\end{lemma}

\begin{proof}
Let $Z=C_p\times C_p\le Z(P)$. For any $H\le G$ that intersects $Z$ trivially,
$HZ/H \cong C_p\times C_p$. By lifting the relation of Example \ref{ex:cpcp}
to $HZ$ and then inducing to $G$, we can replace any occurrence of $H$ in any
$G$-relation
by groups that intersect $Z$ non-trivially, using imprimitive relations.
Each such intersection is normal in $G$, so by Lemma \ref{lem:deflations}
the resulting relation is imprimitive as well.
\end{proof}

We will now develop criteria for a relation to be induced from a subgroup.

\begin{proposition}\label{prop:biggroupinj}
Let $G$ be a finite group and $D\le G$ a subgroup for
which the natural map $B(D)\rightarrow B(G)$ is injective.
If $\Theta=\sum_i n_i H_i$ is a $G$-relation
with $H_i\le D$ for all $i$, then $\Theta$ is induced from a $D$-relation.
\end{proposition}

\begin{proof}
First, we claim that the image of $\Ind:K(D)\rightarrow K(G)$ is a saturated
sublattice 
\footnote{that is, $K(D)\rightarrow K(G)$ has torsion-free cokernel},
i.e. that if $\Theta$ is induced from a $D$-relation and
$R$ is a $G$-relation such that $\Theta = nR$ for some integer $n$, then $R$
is induced from a $D$-relation (and not just from an element of the Burnside
ring of $D$, which is trivially true).
Indeed, it is enough to show that the
image of the induction map $\Ind:K(D)\rightarrow B(G)$ is saturated.
But it is a composition of the two injections
$K(D)\to B(D)\>\smash{\stackrel{\scriptscriptstyle\Ind}{\rightarrow}}\>B(G)$
whose images are clearly saturated, and so it has saturated image.

The image $\cY$ of $\Ind:K(D)\rightarrow K(G)$ is obviously contained in the
space $\cX$
of $G$-relations $\sum_i n_i H_i$ for which $H_i\le D$ for all $i$.
So we only need to compare the ranks of the two spaces.

We have already remarked that the rank of $K(G)$ is equal to the number of
conjugacy classes of non-cyclic subgroups of $G$. A basis for
$K(G)\otimes \bQ$ is obtained by applying Artin's induction theorem to a representative
of each conjugacy class of non-cyclic subgroups of $G$. Hence, it is immediate
that a basis for $\cX\otimes \bQ$ is given by
the subset of this set corresponding to those conjugacy classes of non-cyclic
subgroups that have a representative lying in $D$.
But all these relations are clearly contained in $\cY\otimes \bQ$, so
$\cX\otimes \bQ\subseteq\cY\otimes \bQ\;$and we are done.
\end{proof}

\begin{proposition}\label{prop:normalbig}
Let $G$ be a finite group, and $N\normal G$ a normal subgroup 
of prime index that is either metabelian or supersolvable.
If $\Theta=\sum_i n_i H_i$ is a $G$-relation
with all $H_i\le N$, then $\Theta$ is induced from an $N$-relation.
\end{proposition}

\begin{remark}
It is not true that $\sum_i n_i H_i$ is an $N$-relation, since the $H_i$ are
representatives of $G$-conjugacy classes of subgroups and they might represent
the ``wrong'' $N$-conjugacy classes.
For example, if $H_1$ and ${}^gH_1$ are not conjugate in $N$, then $H_1 - {}^gH_1$
will not be an $N$-relation in general, while it is the zero element in the
Burnside ring of $G$ and in particular a $G$-relation.
\end{remark}

\begin{proof}
Write $p$ for the index of $N$ in $G$, and
fix a generator $T$ of the quotient $G/N\cong C_p$.
Recall (see e.g. \cite{AW} \S8) that for a $C_p$-module $M$,
$$
  H^1(C_p,M) = \frac{\text{1-cocycles}}{\text{1-coboundaries}} = 
    \frac{\ker(1+T+\ldots+T^{p-1})}{\Im(1-T)}.
$$
Let $\Theta = \sum_i n_iH_i$ be a $G$-relation with $H_i\le N$ for all $i$;
we view it as an element of the Burnside ring of $N$.
Write $\tilde\Theta=\sum m_\rho \rho$ for its image 
in the rational representation ring $R_\Q(N)$, the sum taken over the 
irreducible%
\footnote
{Throughout the proof the word `irreducible' refers to 
a rational representation, irreducible over $\Q$.} 
representations of $N$.
Note that 
$\Ind_N^G\tilde\Theta=0$, since $\Theta$ is a $G$-relation.

We need to show that we can add to $\Theta$ a linear combination of terms of the
form ${}^gH-H$ for $H\le N, g\in G$
such that the resulting element of $B(N)$ is an $N$-relation.
In other words, we claim that $\Theta$ is a coboundary for the action
of $G/N$ on $M=B(N)/K(N)$; note that $G$ acts naturally on $B(N)$ and $K(N)$, 
with $N$ acting trivially.

First, observe that the operator $\Res^G_N\Ind_N^G$ on $R_\Q(N)$ is, by
definition of induction, equal to $1+T+\ldots+T^{p-1}$. Since $\Theta$ is a
$G$-relation, $\tilde{\Theta}$ is killed by $\Ind_N^G$, and therefore a fortiori
by $1+T+\ldots+T^{p-1}$. In other words $\tilde\Theta$ is a 1-cocycle under
the action of $C_p$ on the submodule $M$ of $R_\Q(N)$.

It remains to prove that
$$
  H^1(G/N,M)=0.
$$
Any irreducible representation of $N$ is either fixed by $G$ or has
orbit of size $p$. Thus, $R_\Q(N)$ as a $G/N$-module is a direct sum 
of trivial modules $\Z$ and of regular modules $\Z[C_p]$. 
The module $M$, viewed as the image of $B(N)$ in $R_\Q(N)$, is 
of finite index in $R_\Q(N)$ by Artin's induction theorem.
Since $N$ is either metabelian or supersolvable,
a theorem of Berz \cite{Berz, HT} says that
%
$M$ is spanned by elements 
of the form $a_\phi\phi$, as $\phi$ runs over the irreducible 
representations of $N$, for suitable $a_\phi\in\N$. 
Note that $a_\phi=a_{({}^T\phi)}$, because $M\le R_\Q(N)$ is a $C_p$-submodule.
It follows that 
$M$ is also a direct sum of trivial and of regular $C_p$-modules.
%
Now $H^1(C_p,\Z)=\Hom(C_p,\Z)=0$, and also $H^1(C_p,\Z[C_p])=0$ 
since $\Z[C_p]\iso\Hom_{C_p}(\Z[C_p],\Z)$ is co-induced. 
As $H^1$ is additive in direct sums, we get
that $H^1(C_p,M)=0$, as claimed.
\end{proof}

\begin{definition}
A group is called \emph{\qe{p}} if it has a normal cyclic subgroup whose
quotient is a $p$-group.
It is called \emph{quasi-elementary} if it is \qe{p} for some prime $p$.
\end{definition}

\begin{proposition}\label{prop:biggroup}
Let $G$ be a quasi-elementary group with a proper subgroup $D$.
If $\Theta = \sum_i n_iH_i$ is a $G$-relation such that $H_i\le D$ for all
$i$, then it is induced from some proper subgroup of $G$, 
and is in particular imprimitive.
\end{proposition}

\begin{proof}
Write $G=C\rtimes P$, with $P$ a $p$-group and $C$ cyclic of order prime to~$p$.

It suffices to prove the proposition for maximal subgroups $D$ of $G$. Every
maximal subgroup of $G$ is either conjugate to $D=C\rtimes S$ with $S\normal P$
of index $p$, or to $D=U\rtimes P$ where $U$ is a maximal subgroup of $C$.
In the former case, $D\normal G$ is of prime index and is quasi-elementary
and therefore supersolvable, 
so the result follows from Proposition \ref{prop:normalbig}.
Assume that we are in the latter case.
We will show that the map $B(D)\rightarrow B(G)$ is injective,
and the claim will follow from Proposition \ref{prop:biggroupinj}.

In general, the kernel of the induction map $B(D)\rightarrow B(G)$ is generated
by elements of the form $H-{}^gH$ with $H$, ${}^gH\le D$.
We therefore have to verify that two such $H, {}^gH\le  
D=U\rtimes P$ are necessarily $D$-conjugate.

As $U\normal C$ is maximal, $[C:U]=l$ and $G=C_{l^k}D$ for some prime $l$ 
and $k\ge 1$. Write $g=cd, c\in C_{l^k}, d\in D$, so ${}^gH={}^{cd}H$. Replacing
$H$ by ${}^dH$ (which is still a subgroup of $D$), we may assume that
$g=c\in C_{l^k}$. If the order of $c$ is less than $l^k$, then $c\in D$, and we
are done. So assume that $c$ has order $l^k$. If $H$ commutes with $C_{l^k}$,
then $H={}^cH$, and the claim is trivial. Otherwise, there exists $h\in H$
(without loss of generality of order coprime to $l$) for which $hch^{-1}=c^i$
for some $i\not\equiv 1\pmod l$. But then ${}^c hh^{-1}=chc^{-1}h^{-1} = c^{1-i}$
still has order $l^k$, and therefore cannot lie in $D$, contradicting the
assumption that $H,{}^cH\le D$.
\end{proof}

\begin{corollary}\label{cor:smallgpbiggp}
Let $G$ be a quasi-elementary group and let $\{1\}\neq N_j\normal G$, $N_j\le D\lneq G$,
$j=1,\ldots,s$. If $\Theta= \sum_i n_i H_i$ is a
$G$-relation with the property that for each $H_i$
either $N_j\le H_i$ for some $j$ or $H_i\le D$, then $\Theta$ is imprimitive.
\end{corollary}
\begin{proof}
Set $\Theta_0\!=\!\Theta$ and define inductively
$\Theta_j\!=\!\Theta_{j-1}\!-\!N_j\Theta_{j-1}$ for $1\!\le\!j\!\le\!s$.
Then $\Theta_s$ consists only of subgroups of $D$, so it is imprimitive
by Proposition~\ref{prop:biggroup}. Because the projections $N_j\Theta_{j-1}$
are lifted from $G/N_j$, they are also imprimitive.
\end{proof}

\begin{lemma}\label{lem:ideal}
Let $G$ be a finite group and $R$ any $G$-relation, possibly 0.
Then the $\bZ$-span of all imprimitive relations and $R$
is an ideal in the Burnside ring of $G$.
\end{lemma}
\begin{proof}
If $H\neq G$, then $H\cdot \Theta = \Ind^G\Res_H\Theta$ is
imprimitive for any relation~$\Theta$.
If, on the other hand, $H=G$, then $H\cdot \Theta = \Theta$.
\end{proof}

\begin{corollary}\label{cor:alaDeligne}
Let $G$ be a finite group and suppose that there exists an imprimitive
$G$-relation $R$ in which
$G$ enters with coefficient 1. Then $G$ has no primitive relations.
\end{corollary}
\begin{proof}
Write $R\!=\!G-\sum_{H\lneq G} n_H H$. Then $R\cdot \Theta = \Theta-\sum_H n_H
\Ind^G\Res_H\Theta$ for any relation $\Theta$.
By Lemma \ref{lem:ideal}, $R\cdot \Theta$ is imprimitive, and 
$\sum_H n_H \Ind^G\Res_H\Theta$ is also a sum of imprimitive relations.
\end{proof}

\section{A characterisation in terms of quotients}

The main result of this section, Theorem \ref{thm:notquasiel},
gives a characterisation of $\prim(G)$ in terms
of the existence of quasi-elementary quotients of $G$.
First, recall Solomon's induction theorem and a statement complementary to it:

\begin{theorem}[Solomon's induction theorem]\label{thm:Solomon}
Let $G$ be a finite group. There exists a Brauer relation of the form
$G - \sum_H n_H H$ where the sum runs over quasi-elementary subgroups of $G$
and $n_H$ are integers.
\end{theorem}

\begin{proof}
See \cite{Sol} Thm. 1 with
${\mathbf K}=\Q$ or \cite{Isa} Thm. 8.10.
\end{proof}

\begin{theorem}\label{thm:Tims}
Let $G$ be a non-cyclic \qe{p} group.
Then there exists a relation in which $G$ enters with
coefficient $p$. Moreover, in any $G$-relation
the coefficient of $G$ is divisible by $p$.
\end{theorem}

\begin{proof}
See \cite{solomon}.
\end{proof}

\begin{theorem}\label{thm:notquasiel}
Let $G$ be a non-quasi-elementary group.
\begin{enumerate}
\item $\Prim(G)\iso \Z$ if all proper quotients of $G$ are cyclic.
\item $\Prim(G)\iso \Z/p\Z$ if all proper quotients of $G$ are \qe{p}
for the same prime $p$, and at least one of them is not cyclic.
\item $\Prim(G)=0$ otherwise.
\end{enumerate}
In cases (1) and (2), $\Prim(G)$ is generated by any relation in which
$G$ has coefficient~1.
\end{theorem}

\begin{proof}
By Solomon's induction theorem, $G$ has a relation of the form
$R= G - \sum_{H\ne G}n_H H$, and we claim that
$R$ generates $\prim(G)$ in all cases.
By Lemma \ref{lem:ideal}, the span $I$ of the set of imprimitive
relations and of $R$ is an ideal in $B(G)$. To show that $K(G)\subset I$,
let $\Theta$ be any relation.
Then $\Theta = R\cdot\Theta + (\Theta-R\cdot \Theta)$ and
$R\cdot\Theta\in I$. Also,
$$
  \Theta-R\cdot \Theta = \sum_{H\neq G} n_H(\Theta\cdot H)
$$
is imprimitive and therefore also in $I$. So $\Theta\in I$, as claimed.

It remains to determine the smallest integer $n>0$ such that $G$ has
an imprimitive relation of the form $\Theta=nG-\sum_{H\ne G}m_H H$.
Then $\prim(G)\iso\Z/n\Z$ (and $\Z$ if there is no such $n$).
Clearly $G$ does not enter the relations that are induced from proper
subgroups, so such a $\Theta$ must be a linear combination of relations
lifted from proper quotients.

(1) If all proper quotients of $G$ are cyclic, there are no such relations.

(2) If all proper quotients are \qe{p}, then $n$ is a multiple of $p$
by Theorem \ref{thm:Tims}, and there is a relation with $n=p$ by the same
theorem if one of these quotients is not cyclic.

(3) Otherwise, either
\begin{enumerate}
\item[(a)]
some proper quotient $G/N$ is not quasi-elementary,
in which case we apply Solomon's induction to $G/N$ and lift
the resulting relation to $G$; or
\item[(b)] $G$ has two proper non-cyclic
quotients $G/N_1, G/N_2$ which are $p$- and \qe{q} with $p\ne q$,
in which case we take a linear combination of the two lifted relations
$pG+...$ and $qG+...$.
\end{enumerate}
In both cases, there is an imprimitive relation
with $n\!=\!1$, so $\Prim(G)\!=\!0$.
\end{proof}

\begin{corollary}\label{cor:quasielquotients}
If a finite group $G$ has a primitive relation, then there is a prime $p$
such that every proper quotient of $G$ is $p$-quasi-elementary.
\end{corollary}

\begin{proof}
If $G$ itself is $p$-quasi-elementary, then so are all its quotients, and
there is nothing to prove. Otherwise, apply the theorem.
\end{proof}

\begin{corollary}
\label{cor:ses}
Let $G$ be a finite group that has a primitive relation.
Then $G$ is an extension of the form
\begin{eqnarray}\label{ses}
  1 \rightarrow S^d \rightarrow G \rightarrow Q\rightarrow 1, \qquad\qquad d\ge 1
\end{eqnarray}
with $S$ a simple group and $Q$ \qe{p}.
Moreover, if $S$ is not cyclic (equivalently if $G$
is not soluble), then the canonical map $Q\rightarrow \Out(S^d)$
is injective and $S^d$ has no proper non-trivial subgroups
that are normal in $G$.
In this case, $\prim(G)\cong \bZ$ if $Q$ is cyclic and
$\prim(G)\cong \bZ/p\bZ$ otherwise.
\end{corollary}

\begin{proof}
By the existence of chief series for finite groups, any $G\ne\{1\}$ is
an extension \eqref{ses} of some group $Q$, with simple $S$.
Because $G$ has a primitive relation, $Q$ is quasi-elementary by Theorem
\ref{thm:notquasiel}.

Now suppose $S$ is not cyclic, and consider the kernel $K$ of the map
$G\to \Aut(S^d)$ given by conjugation. The centre of $S^d$ is trivial,
so $K\cap S^d=\{1\}$. If $K\ne\{1\}$, then $G/K$ is a proper
non-quasi-elementary quotient, contradicting Theorem \ref{thm:notquasiel}. So
$G\injects \Aut(S^d)$ and, factoring out $S^d\iso\Inn(S^d)$, we get
$Q\injects \Out(S^d)$. In the same way, if $N\normal G$
is a proper subgroup of $S^d$, then $G/N$ is not quasi-elementary, so
again $N=\{1\}$.

Finally, the description of $\prim(G)$ is given by Theorem \ref{thm:notquasiel}.
\end{proof}

\begin{remark}
\label{rem:nonsolu}
Conversely, suppose that $G$ is an extension as in \eqref{ses}
with \qe{p} $Q$, non-cyclic $S$ and $Q\injects \Out(S^d)$.
Suppose also that $S^d$ has no proper non-trivial subgroups
that are normal in $G$. It follows that every non-trivial normal subgroup
of $G$ contains $S^d$. So $G$ is not quasi-elementary
but every proper quotient of it is $p$-quasi-elementary, 
and therefore $G$ has a primitive relation.
This proves Theorem \ref{thm:A} for all non-soluble groups.
\end{remark}

\section{Primitive relations in $p$-groups}\label{sec:pgroups}

\begin{definition}
The normal $p$-rank of a finite group $G$ is the maximum of the ranks of the elementary abelian normal
$p$-subgroups of $G$.
\end{definition}

As in Bouc's work \cite{Bouc}, the groups of normal $p$-rank one will be of particular importance to us.
We will repeatedly need the following classification:

\begin{proposition}[\cite{Gor-68}, Ch. 5, Thm. 4.10]\label{prop:prankone}
Let $P$ be a $p$-group with normal $p$-rank one. Then $P$ is one of the following:
\begin{itemize}
\item the cyclic group $C_{p^n}\!=\!\langle c|c^{p^n}\!=\!1\rangle$;
\item the dihedral group $D_{2^{n+1}}\!=\!\langle c,x|c^{2^n}\!=\!x^2\!=\!1, xcx\!=\!c^{-1}\rangle$ with $n\ge 3$;
\item the generalised quaternion group, $Q_{2^{n+2}} \!=\! \langle c,x|c^{2^n}\!=\!x^2,x^{-1}cx \!=\! c^{-1}\rangle$ with $n\ge 1$;
\item the semi-dihedral group $SD_{2^{n+1}}\!=\!\langle c,x|c^{2^n} \!=\! x^2 \!=\! 1, xcx\!=\!c^{2^{n-1}-1}\rangle$ with $n\ge 3$.
\end{itemize}
\end{proposition}

We now present an alternative proof of the Tornehave--Bouc theorem
(\cite{Bouc}, Cor. 6.16). The ingredients are the results of
\S\ref{s:imprim} and a lemma of Bouc \cite[Lemma 6.15]{Bouc}.

\begin{theorem}[Tornehave--Bouc]\label{thm:pgroups}
  All Brauer relations in $p$-groups are $\Z$-linear
  combinations of ones lifted from subquotients $P$ of the following types:
  \begin{enumerate}[(i)]
    \item\label{item:cpcp}
    $P\iso C_p\times C_p$ with the relation
    $1 - {\sum_C C} + p\cdot P$, the sum taken over
    all subgroups of order $p$;
    \item\label{item:Heisenberg}
    $P$ is the Heisenberg group of order $p^3$ (which is isomorphic to $D_8$
    when $p=2$), and the relation is
    $I - IZ - J + JZ$ where $Z=Z(P)$ and $I$ and $J$ are two
    non-conjugate non-central subgroups of order $p$;
    \item\label{item:dih}
    $P\iso D_{2^{n}}$, $n\geq 4$, with the relation
    $I - IZ - J + JZ$, where $Z=Z(P)$
    and $I$ and $J$ are two non-conjugate non-central subgroups of order 2.
  \end{enumerate}
\end{theorem}

\begin{proof}
Let $P$ be a $p$-group that has a primitive relation.
By Lemma \ref{lem:noncycliccentre}, either
$P=C_p\times C_p$ or $P$ has cyclic centre.
The former is covered by Example \ref{ex:cpcp}, so
assume that we are in the latter case, and let $\centralCp$
be the unique central subgroup of order $p$.

First, suppose $P$ has normal $p$-rank $r\ge 2$,
with $V=(C_p)^r\normal P$. The conjugation action of $P$ on $V$
is upper-triangular, as is any action of a $p$-group on an
$\F_p$-vector space. So there are normal subgroups $(C_p)^j\normal G$
for all $j\le r$, and we denote by $E$ one for $j=2$.
Note that $\centralCp\subset E$, since any normal subgroup of a $p$-group
meets its centre. By \cite[Lemma 6.15]{Bouc}, any occurrence in a relation
of a subgroup that does not contain $\centralCp$ and is not contained
in the centraliser $C_P(E)$ of $E$ in $P$ can be replaced by
subgroups that either contain $\centralCp$ or are contained in $C_P(E)$,
using a relation from a subquotient isomorphic to the Heisenberg
group of order $p^3$. The remaining relation is then
imprimitive by Corollary \ref{cor:smallgpbiggp}.
So $P$ has a primitive relation if and only if it is
the Heisenberg group of order $p^3$.

Now suppose that $r=1$, so $P$ is as in \ref{prop:prankone}.
If $P$ is cyclic or
generalised quaternion, then every non-trivial subgroup contains $\centralCp$, so
$P$ has no primitive relations by Corollary \ref{cor:smallgpbiggp}. If $P$ is
semi-dihedral, then the only conjugacy class of non-trivial subgroups of
$P$ that do not contain $\centralC2$ is that of non-central involutions,
represented by $\langle x\rangle$, say. But $x$ and $\centralC2$ generate a proper
subgroup of $P$, so $P$ again
has no primitive relations by Corollary \ref{cor:smallgpbiggp}.
Finally, if $P$ is dihedral of order $2^n$, $n\geq 4$, then
there are two conjugacy classes of non-trivial subgroups that do not
contain $\centralC2$, represented, say, by $I$ and $J$.
Using the relation in (\ref{item:dih}) (cf. \cite[page 25]{Bouc})
any occurrence of $I$ in a relation can be replaced by $J$ and by subgroups
that contain $\centralC2$. In the resulting relation, every subgroup will
either contain $\centralC2$ or will be contained in
$D=\centralC2\times J$, which is a proper subgroup of $P$.
So, applying Corollary \ref{cor:smallgpbiggp} again,
we see that the group of primitive relations of $P$
is generated by the relation of
(\ref{item:dih}) and the theorem is proved.
\end{proof}

\section{Main reduction in soluble groups}
\label{sec:solvable}

\begin{theorem}\label{thm:mainred}
Every finite soluble group that has a primitive relation is either
\begin{enumerate}[(i)]
\item quasi-elementary, or
\item\label{item:serre} of the form $(C_l)^d\rtimes Q$,
  where $l$ is a prime, $Q$ is quasi-elementary~and~acts faithfully and
  irreducibly on the $\bF_l$-vector space $(C_l)^d$ by conjugation,~or
\item\label{item:almostserre}
  of the form $(C_l\rtimes P_1)\times (C_l\rtimes P_2)$,
  where $l,p$ are primes, and $P_i\injects\Aut C_l$ are cyclic $p$-groups.
\end{enumerate}
\end{theorem}

\begin{proof}
Since $G$ is soluble and has a primitive relation,
by Corollary \ref{cor:ses} it is an extension of the form
\begin{equation}
\label{sess}
  1 \rightarrow (C_l)^d \rightarrow G \rightarrow Q\rightarrow 1, \qquad\qquad d\ge 1,
\end{equation}
with $Q$ quasi-elementary. 
We may assume $d\ge 1$ (otherwise we are in (i)) and $Q\ne\{1\}$
(otherwise $G\iso C_l\times C_l$, e.g. by Theorem \ref{thm:pgroups}, and 
we are in~(iii)).
Consider the various possibilities for
the structure of $Q$ and its action on $W=(C_l)^d$ by conjugation.

{\bf (A) }\emph{Suppose that $l$ does not divide $|Q|$}.
The sequence \eqref{sess} then splits by the Schur--Zassenhaus theorem,
so $G=W\rtimes Q$. The kernel of the action of $Q$ on $W$ is then
a normal subgroup $N\normal G$.

\textbf{Case 1: $N\ne\{1\}$ and $Q$ is cyclic.}
By Corollary \ref{cor:quasielquotients}, $G/N$ is quasi-elementary.
If it is \qe{p} for some $p\neq l$,
then its $l$-part must be
cyclic, so $d=1$. Moreover, since $Q/N$ acts faithfully on $C_l$, it must be a
$p$-group. So, writing $Q = Q_p\times Q_{p'}$, where $Q_p$ is the Sylow $p$-subgroup
of $Q$, we deduce that $N$ contains $Q_{p'}$, which is cyclic of order coprime to $l$,
and so $G=(C_l\times Q_{p'})\rtimes Q_p$ is quasi-elementary (case~(i)).
If $G/N$ is \qe{l}, then $l\nmid|Q|$ implies that $Q/N\normal G/N$,
so $G/N=(Q/N)\times W$. But $N$ is the whole kernel of the action of $Q$
on $W$, so $Q/N$ must be trivial. In this case $Q=N$ is normal in $G$, and
$G=Q\times W$ is again quasi-elementary.

\textbf{Case 2: $N\ne\{1\}$ and $Q$ is not cyclic.}
Write $Q=C\rtimes P$ with $C$ cyclic of order coprime to $lp$
and $P$ a $p$-group.
This time, we know that $G/N$ is \qe{p} by Corollary \ref{cor:quasielquotients}.
Since $p\neq l$, we have $d=1$. Also,
because $G/N$ is \qe{p} and the action of $Q/N$ on $C_l$ is faithful,
$Q/N$ must be a $p$-group.
So $N$ contains $C$, and $G=(C_l\times C)\rtimes P$ is \qe{p}.

\textbf{Case 3: $N=\{1\}$ and $Q$ acts reducibly.}
Since $l\nmid |Q|$, the $\F_l$-representation $W$ of $Q$ is completely reducible.
Say $W=\bigoplus_{i=1}^n V_i$ with irreducible $V_i$; so $V_i\normal G$.

Let $p$ be a prime divisor of $|Q|$. A Sylow $p$-subgroup of $Q$
acts faithfully on $W$, so it acts non-trivially on one of the
$V_i$, say on $V_1$.
Because $U=G/(V_2\oplus\cdots\oplus V_n)\iso V_1\rtimes Q$
is quasi-elementary by Corollary \ref{cor:quasielquotients}, and because its $p$-Sylow
is not normal in it, $U$ must be $p$-quasi-elementary (and not cyclic).
However, Corollary \ref{cor:quasielquotients} asserts that
all proper non-cyclic quotients of $G$ are quasi-elementary with respect
to the same prime, so $|Q|$ cannot have more than one prime divisor.
In other words, $Q$ is a $p$-group.

Now, both $G/V_1$ and $G/V_2$ must be \qe{p}, so their $l$-parts are cyclic.
This is only possible if $n=2$ and $\dim V_1=\dim V_2=1$.
So $W=C_l\times C_l$, and
$$
   Q\>\injects\>(\Aut C_l)\times(\Aut C_l) \>\iso\> \F_l^\times\times\F_l^\times
$$
is an abelian $p$-group.
This is case \eqref{item:almostserre} of the theorem.

\item \textbf{Case 4: $N=\{1\}$ and $Q$ acts irreducibly.}
This is case (\ref{item:serre}).

{\bf (B) }\emph{Suppose that $l$ divides $|Q|$.}

\textbf{Case 5: $Q$ is \qe{l'} for $l'\neq l$}. 
Let $\bar L$ be a Sylow $l$-subgroup of $Q$. Since $l'\ne l$, $\bar L$ is 
cyclic and normal in $Q$, and we write $L\normal G$ for its inverse image in $G$. 
So $G$ is an extension of $\tilde Q=Q/\bar L$ by $L$. By the Schur-Zassenhaus
theorem it is a split extension, and we may view $\tilde Q$ as a subgroup 
of $G$ and consider its conjugation action on $L$.

If the Frattini subgroup $\Phi(L)$ is trivial, then $L\iso (C_l)^m$ for some $m$
and we are back in case (A) of the proof. So suppose that $\Phi(L)\ne\{1\}$. 
Then $G/\Phi(L)$ is quasi-elementary by Corollary \ref{cor:quasielquotients}.

Assume first that $G/\Phi(L)$ is $p$-quasi-elementary for $p\ne l$. 
Then $L/\Phi(L)$ must be cyclic, hence $L$ is cyclic (by a standard property
of $l$-groups). Moreover, $\tilde Q= R\rtimes P$ with $R$ cyclic and $P$ a 
$p$-group, and $G/\Phi(L)=(L/\Phi(L)\times R)\rtimes P$. Now $R$ acts trivially
on $L/\Phi(L)$ and has order prime to~$l$, so $R$ acts trivially on
$L$ by the classical theorem of Burnside 
that the kernel of $\Aut(L)\to \Aut(L/\Phi(L))$ is an $l$-group
(\cite{Gor-68} Ch. 5, Thm. 1.4). It follows that 
$G=(L\times R)\rtimes P$ and $L\times R$ is cyclic, so $G$ is 
$p$-quasi-elementary.

Assume that $G/\Phi(L)$ is $l$-quasi-elementary. Then $\tilde Q$ must be 
cyclic and normal in $G/\Phi(L)$, and therefore $G/\Phi(L)=L/\Phi(L)\times\tilde Q$.
Again $\tilde Q$ acts trivially on $L/\Phi(L)$, hence on $L$ by Burnside's
theorem. It follows that $G=L\times \tilde Q$ is $l$-quasi-elementary.

\textbf{Case 6: $Q$ is non-cyclic \qe{l}.}
Now $Q=C\rtimes P$ with $C$ cyclic of order prime to $l$, and $P$ an $l$-group,
both non-trivial. By Schur-Zassenhaus we may view $C$ as a subgroup of $G$.
We may also assume that $C$ acts non-trivially on $W$, for otherwise 
$C\times W$ is a normal subgroup of $G$ in which $C$ is characteristic,
so $C\normal G$ and $G$ is quasi-elementary. 

Since $|C|$ and $|W|$ are coprime,
$W$ is completely reducible as a representation of $C$ over $\F_l$.
Therefore, the invariant subspace $W^C$ has a (non-trivial) complement on which 
$C$ acts faithfully. 
Since $W^C$ is a $P$-representation, it is a normal subgroup of $G$. If it is 
non-zero, then $G/W^C$ is $l$-quasi-elementary by Corollary \ref{cor:quasielquotients},
so the image of $C$ is normal in it. But so is the image of $W$, so the two 
commute, contradicting the faithfullness of the action of $C$ on $W/W^C$. 
In other words, $W^C=0$.

Now the inflation-restriction sequence for $C\normal Q$ acting on $W$ reads
$$
  H^2(Q/C,W^C) \lar H^2(Q,W) \lar H^2(C,W).
$$
The first group is zero as $W^C=0$, and the last one is zero as it is killed
by $|C|$ and by $|W|$, which are coprime. So the middle group,
which classifies extensions of $Q$ by
$W$ up to splitting, is zero, in other words 
$G=W\rtimes Q$ is a split extension.

Next, we show that $W$ is irreducible as a representation of $Q$. If not,
let $0\subsetneq V\subsetneq W$ be a
subrepresentation. Since $G/V$ is \qe{l}
(Corollary \ref{cor:quasielquotients} again), $C$ must act trivially on $W/V$.
But, using complete reducibility again, this contradicts the fact $W^C=0$.

Finally, consider the kernel $N$ of the action of $Q$ on $W$.
As $G$ is a split extension, $N$ may be viewed as a (normal) subgroup of $G$. 
If $N$ is non-trivial, $G/N$ is \qe{l}, and so
$CN/N\normal G/N$, which implies $CN\normal G$. Moreover, the commutators
$[C,W]$ lie both in $W$ and $CN$, hence in $W\cap CN=\{1\}$. Therefore, 
$W$ centralises $C$, so $C$ is normal in $G$, and it follows that $G$ is \qe{l}.
%
If, on the other hand,
$N$ is trivial, then we are in case~\eqref{item:serre}.
\end{proof}

\begin{remark}\label{rmrk:charSemiDir}
Before continuing, we recall from \cite[\S8.2]{SerLi} the classification
of irreducible characters of semi-direct products by abelian groups.
Let $G=A\rtimes H$
with $A$ abelian. The group $H$ acts on 1-dimensional characters of $A$ via
$$
  h(\chi)(a) = \chi(hah^{-1}),\qquad h\in H,\;a\in A,\;\chi:A\rightarrow \bC^\times.
$$
Let $X$ be a set of representatives of $H$-orbits of these characters.
For $\chi\in X$ write $H_\chi$ for its stabiliser in $H$.
Then $\chi$ can be extended to a one-dimensional
character of its stabiliser $S_\chi = A\rtimes H_\chi$ in $G$ by defining it
to be trivial on $H_\chi$.
Let $\rho$ be an irreducible character of $H_\chi\cong S_\chi/A$ and
lift it to $S_\chi$.
Then $\Ind^G_{S_\chi}(\chi\otimes\rho)$ is an irreducible character of $G$
and all irreducible characters of $G$ arise uniquely in this way,
for varying $\chi\in X$ and $\rho$.
\end{remark}

\begin{proposition}\label{prop:SerreGroups}
Let $G=W\rtimes H$ with $W\iso(C_l)^d$ for $d\ge 2$, and $H$ acting faithfully
on $W$.
Let $\cU$ be a set of representatives of the
$G$-conjugacy classes of hyperplanes $U\subset W$,
and write $H_U=N_H(U)$ for $U\in \cU$. Then
$$
  \Theta = G - H + \sum_{U\in\cU}(H_UU - H_UW)
$$
is a $G$-relation.
\end{proposition}
\begin{proof}
We retain the notation of Remark \ref{rmrk:charSemiDir} for the irreducible
characters of $G$.
Choose the set $X$ of representatives for the $H$-orbits of 1-dimensional
characters of $W$ in such a way that $\ker\chi\in\cU$ for $\triv\ne\chi\in X$.

To prove that $\Theta$ is a relation, it suffices to show that
\beq
  \C[G/H]\ominus\triv &=&
  \displaystyle
  \bigoplus_{\scriptscriptstyle\chi\in X, \chi\ne\triv}
  \>\>\Ind^G_{S_\chi}(\chi\otimes\triv_{H_\chi}),\cr
  \C[G/H_UU]\ominus\C[G/H_UW] &=&
  \displaystyle
  \!\!\!\bigoplus_{\scriptscriptstyle{\chi\in X,\ker\chi=U}}
  \!\Ind^G_{S_\chi}(\chi\otimes\triv_{H_\chi})
  \qquad\text{for }U\in\cU.
\eeq
To do this, first compute the decomposition of $\C[G/T]$ into irreducible
characters for an arbitrary $T\<G$. The multiplicity $\mm T\chi\rho$
of $\Ind^G_{S_\chi}(\chi\otimes\rho)$ in $\C[G/T]$ is
\beq
\mm T\chi\rho &=
\>\>\langle \Ind_{S_\chi}^G(\chi\!\tensor\!\rho),\Ind^G{\triv_T}\rangle_G =
  \langle \Res_T\Ind_{S_\chi}^G(\chi\!\tensor\!\rho),\triv\rangle_T\\[4pt]
  &\displaystyle=\!\!\!\sum_{\scriptscriptstyle x\in S_\chi\backslash G/T}\langle\Ind^{{}^xT}\Res^{S_\chi}_{S_\chi\cap {}^xT}(\chi\!\tensor\!\rho),\triv\rangle_{{}^xT}\\[4pt]
  &\displaystyle=\!\!\!\sum_{{\scriptscriptstyle x\in S_\chi\backslash G/T}}\langle\Res_{S_\chi\cap {}^xT}(\chi\!\tensor\!\rho),\triv\rangle_{S_\chi\cap {}^xT}.
\eeq
Next, take $T=H$. Since $W\subseteq S_\chi$ for each $\chi\in X$,
there is a unique double coset in
$S_\chi\backslash G/H$, the trivial one. So
$$
  \mm H\chi\rho = \langle\Res^{S_\chi}_{H_\chi}(\chi\otimes\rho),\triv_{H_\chi}\rangle_{H_\chi} =
  \langle\rho,\triv_{H_\chi}\rangle_{H_\chi} =
  \left\{
  \begin{array}{ll}1,&\rho=\triv\cr0,&\rho\ne\triv,\end{array}
  \right.
$$
as claimed.
Finally, for $U\in\cU$ we compare $\mm{H_UU}\chi\rho$ and $\mm{H_UW}\chi\rho$.

If $\chi=\triv$ and $\rho$ is an irreducible representation of $G/W$
lifted to $G$, then
\beq
  \mm{H_UU}\chi\rho &=& \langle\Ind^G\triv_{H_UU},\rho\rangle_G
  = \langle \triv,\Res_{H_UU}\rho\rangle_{H_UU}\\[2pt]
  &=& \langle \triv,\Res_{H_U}\rho\rangle_{H_U}
  = \langle\Ind^G\triv_{H_UW},\rho\rangle_G
  = \mm{H_UW}\chi\rho.
\eeq
For $\chi\ne \triv$,
$$
  \mm{H_UU}\chi\rho \>= \!\!\!\!
  \sum_{\smash{x\in S_\chi\backslash G/H_UU}}\bigl\langle
    \Res_{S_\chi\cap {}^x(H_UU)}(\chi\otimes\rho),
    \triv \bigr\rangle_{S_\chi\cap {}^x(H_UU)}.
$$
If $\ker\chi\neq U$, or if $\ker\chi = U$ but $x$ represents a non-trivial
double coset, then the corresponding summand is 0, since
$S_\chi\cap {}^x(H_UU)$ contains ${}^xU$, a hyperplane of $W$ distinct from
$\ker\chi$, and the restriction to this hyperplane is a sum of several
copies of one non-trivial character. The same is true for $H_UW$.
If, on the other hand, $\ker\chi = U$, then $H_\chi \le H_U$, so that
$S_\chi \cap H_UU = H_\chi U$. Therefore
$$
  \mm{H_UU}\chi\rho =
  \left\{
  \begin{array}{ll}1,&\rho=\triv\cr0,&\rho\ne\triv\end{array}
  \right. \qquad\text{and}\qquad \mm{H_UW}\chi\rho=0.
$$
\end{proof}

\begin{proposition}\label{prop:G20}
Let $G=C_l\rtimes H$, with $l$ prime and $H\ne\{1\}$ acting
faithfully on $C_l$. Then $\prim(G)\cong \bZ$.
If $H\cong C_{p^{k}}$ is of prime-power order, then $\prim(G)$ is generated by
$$
  C_{p^{k-1}}-pH-C_l\rtimes C_{p^{k-1}} + pG.
$$
If $H\cong C_{mn}$ with coprime $m,n>1$, then $\prim(G)$ is generated by
$$
  G-H+\alpha(C_n-C_l\rtimes C_n)+\beta(C_m-C_l\rtimes C_m),
$$
where $\alpha m+\beta n=1$.
\end{proposition}
\begin{proof}
The existence of the two relations follows immediately from Example
\ref{exer:G20},
applied to $\tilde{H}=C_m\<H$ and $\tilde{H}=C_n\<H$.
If $H$ has composite order, the result follows from
Theorem \ref{thm:notquasiel}, case (1). If $H\iso C_{p^k}$, then $G$ is \qe{p},
so the coefficient
of $G$ in any relation is divisible by $p$ by Theorem \ref{thm:Tims}. Clearly,
no relation in which $G$ enters with non-zero coefficient can be induced from a
subgroup. But also, no such
relation can be lifted from a proper quotient, since all proper quotients of
$G$ are cyclic and therefore have no non-trivial relations.
\end{proof}

\begin{corollary}
Theorem \ref{thm:A} holds for all finite non-quasi-elementary
groups.
\end{corollary}

\begin{proof}
The theorem is already proved for non-soluble groups (Remark \ref{rem:nonsolu}),
so suppose $G$ is soluble but not quasi-elementary.
Then, if $G$ has a primitive relation, it falls under \eqref{item:serre}
or \eqref{item:almostserre} of Theorem \ref{thm:mainred}.
This gives one direction.

Conversely, suppose $G$ is of one of these two types, in particular
$G\iso(C_l)^d\rtimes Q$, with $Q$ quasi-elementary
and acting faithfully on $(C_l)^d$ by conjugation.
It is easy to see that
every proper quotient of $G$ is quasi-elementary.
So Theorem \ref{thm:notquasiel} combined with
Proposition \ref{prop:SerreGroups} for $d\ge 2$
and Proposition \ref{prop:G20} for $d=1$
give the asserted description of $\Prim(G)$.
\end{proof}

\section{Quasi-elementary groups}\label{sec:qegen}

In this section, we determine the structure
and the representatives of $\prim(G)$ for quasi-elementary groups
that are not $p$-groups. This is case (4) of Theorem \ref{thm:A},
and it is by far the most difficult one.

\begin{notation}
For the rest of the section we fix
\vspace{\baselineskip}\newline
\begin{tabular}{l | l}
$P$ & a non-trivial $p$-group,\\
$C$ & a non-trivial cyclic group of order coprime to $p$,\\
$G=C\rtimes P$ & a quasi-elementary group with normal subgroup $C$\\ &
and a fixed complementary subgroup $P\le G$,\\
$K\normal P$ & the kernel of the conjugation action of $P$ on $C$.
\end{tabular}
\end{notation}
We begin by showing that the presence of primitive relations forces tight
restrictions on the structure of $K$. We then
write down generators for $\Prim(G)$ and give
necessary and sufficient group-theoretic criteria for these relations
to be primitive.

\subsection{The kernel of the conjugation action}\label{sec:kernel}

\begin{lemma}\label{lem:StructOfP}
If $P$ has normal $p$-rank one or is isomorphic to $D_8$, and $K\ne\{1\}$,
then $G$ has no primitive relations.
\end{lemma}

\begin{proof}
By Proposition \ref{prop:prankone}, $P$ is either cyclic, generalised quaternion,
semi-dihedral, or dihedral. We will consider these cases separately.
We may assume that $P\not\iso C_p$, for otherwise
$K=P$ and \hbox{$G=P\times C$} is cyclic. In the remaining cases, we use
the notation of Proposition \ref{prop:prankone} for the generators $c,x$ of $P$.
Denote by $\centralCp$ the unique
central subgroup of $P$ of order $p$. Note that $K$ contains $\centralCp$,
since any normal subgroup of a $p$-group intersects its centre non-trivially.

If \emph{$P$ is cyclic or generalised quaternion,} then every
non-trivial subgroup of $P$ contains $\centralCp$. So every subgroup of $G$
either contains $\centralCp$, or contains a non-trivial subgroup of $C$,
or is contained in $D=\centralCp\times C\normal G$.
By Corollary~\ref{cor:smallgpbiggp}, $G$ has no primitive relations.

If \emph{$P$ is semi-dihedral}, then there is only one conjugacy
class of subgroups of $P$ that do not contain $\centralC2$, represented by
$\langle x\rangle$.
Now, up to conjugation, every subgroup of $G$
either contains $\centralC2$ or a non-trivial subgroup of $C$, or is contained
in $D=C\rtimes(\centralC2\times \langle x\rangle )\normal G$.
By Corollary \ref{cor:smallgpbiggp}, we are done.

If \emph{$P$ is dihedral}, then there are two conjugacy classes of non-trivial
subgroups of $P$ that intersect $\langle c\rangle$ trivially, $I$ and $J$,
say. They are each generated by a non-central involution.
There is a $P$-relation (cf. Theorem \ref{thm:pgroups})
\[
  I - J - I\centralC2 + J\centralC2.
\]
Thus, any occurrence of $I$ in any relation can be replaced by groups that
either contain $\centralC2$ or are contained in $J\centralC2$, using a relation
that is induced from $P$, which is a proper subgroup of $G$. Similarly, any
occurrence of $\tilde{C}\rtimes I$ for $\tilde{C}\le C$ can be replaced by
subgroups that either contain $\centralC2$ or
are contained in $C\rtimes J\centralC2$ using a relation from a proper subquotient.
In summary, by adding imprimitive relations to any given $G$-relation, all
subgroups can be arranged to either contain $\centralC2$ or be contained in
$C\rtimes J\centralC2$ and we are again done by Corollary \ref{cor:smallgpbiggp}.
\end{proof}

\begin{lemma}\label{lem:normalnoncentral}
Suppose $P$ has a non-central normal subgroup $E\iso C_p\times C_p$
that intersects $K$ non-trivially. Then $G$ has no primitive relations.
\end{lemma}

\begin{proof}

Since $E\normal P$, the intersection $U=E\cap Z(P)$ is non-trivial.
By assumption, $U$ is not the whole of $E$, so $C_p\iso U\normal P$,
and the action of $P$ on $E$ by conjugation factors through a group
$\smallmatrix 1*01$ of order $p$. In particular, no other $C_p\<E$ except
for $U$ is normal in $P$, so every normal subgroup of $P$ that meets
$E$ non-trivially must contain $U$; hence $U\subset K$.
So $U$ commutes both with $C$ and with $P$, in particular $U\normal G$.

The centraliser $C_P(E)$ of $E$ in $P$ has index $p$ in $P$.
By \cite[Lemma 6.15]{Bouc}, if $H$ is any subgroup of $P$ that does not
contain $U$ and is not contained in $C_P(E)$, then any occurrence
of $H$ in a relation can be replaced by subgroups that either contain
$U\le Z(G)$ or are contained in $C_P(E)$ using a relation induced
from $P$, which is a proper subgroup of $G$. Similarly, any group of the form
$\tilde{C}\rtimes H$ for $\tilde{C}\le C$ and $H$ as above can be
replaced by subgroups that either contain $U$ or are contained
in $D=C\rtimes C_P(E)$ using a relation from the quotient $G/\tilde{C}$.
By Corollary \ref{cor:smallgpbiggp}, $G$ has no primitive relations.
\end{proof}

\begin{corollary}\label{cor:ZP}
If $K\ne\{1\}$ and $P$ has cyclic centre, then $G$ has no primitive relations.
\end{corollary}
\begin{proof}
If $P$ has normal $p$-rank one, we are done by Lemma \ref{lem:StructOfP}.
Otherwise $P$ has a normal subgroup $E\iso C_p\times C_p$
(cf. proof of \ref{thm:pgroups}). Since $Z(P)$ is cyclic, $E$ is not central.
Also, both $E$ and $K$ intersect $Z(P)$ non-trivially, so they both contain
the unique $C_p\le Z(P)$, and thus $G$ has no primitive relations by
Lemma \ref{lem:normalnoncentral}.
\end{proof}

\begin{lemma}
\label{lem:herzog}
Let $T$ be any $p$-group. Then either $T=\{1\}$ or $T\iso D_8$
or $T$ has normal $p$-rank one or the number of normal subgroups
of $T$ isomorphic to $C_p \times C_p$ is congruent to 1 modulo $p$.
\end{lemma}

\begin{proof}
By a Theorem of Herzog \cite[Theorem 3]{Her},
the number $\alpha$ of elements in
$T$ of order $p$ is congruent to $-1$ modulo $p^2$ if and only if $T\not\cong D_8$
and has normal $p$-rank greater than one.
We consider two cases:

\textbf{Case 1:} $Z(T)$ is cyclic. Since every normal subgroup of $T$ intersects
the centre non-trivially and
since there is a unique subgroup $\langle z\rangle$ of order $p$ in the centre,
any normal $C_p\times C_p$ is generated by $z$ and a non-central element $a$ of
order $p$. For an arbitrary non-central element $a$ of order $p$,
$\langle a,z\rangle$ need not be normal, but the size of its orbit under
conjugation is a power of $p$. So the number of normal such $C_p\times C_p$
is congruent modulo $p$ to the number of all $C_p\times C_p$ that intersect the
centre non-trivially. Finally, $p^2-p$ different non-central elements define
the same subgroup, so the number $\beta$ of normal subgroups isomorphic to
$C_p\times C_p$ is congruent to $(\alpha-(p-1))/(p^2-p)$ modulo $p$. Thus,
\beq
  T\not\cong D_8 \text{ and } \exists \;C_p\times C_p\normal T & \Leftrightarrow & \alpha \equiv -1 \pmod{p^2}\\
  & \Leftrightarrow & \alpha - p +1 \equiv -p \pmod{p^2}\\
  & \Leftrightarrow & \beta = \frac{\alpha-(p-1)}{p^2-p} \equiv 1 \pmod{p},
\eeq
as required.

\textbf{Case 2:} $Z(T)$ is not cyclic. Then a normal subgroup of $T$ isomorphic
to $C_p\times C_p$ is either contained in $Z(T)$ or intersects it in a line.
Let $Z(T)$ have normal $p$-rank $r\ge 2$. Any $C_p\times C_p \le Z(T)$ is
generated by two linearly independent elements of order $p$ and there
are $(p^r-1)(p^r-p)/2$ unordered pairs of such elements. Each $C_p\times C_p$
contains $(p^2-1)(p^2-p)/2$ pairs and so there are
\[
\frac{(p^r-1)(p^r-p)}{(p^2-1)(p^2-p)}=\frac{(p^r-1)(p^{r-1}-1)}{(p^2-1)(p-1)}\equiv 1\pmod{p}
\]
distinct subgroups of $Z(T)$ that are isomorphic to $C_p\times C_p$.
Since there are $p^r-1 \equiv -1\pmod{p^2}$ elements in $Z(T)$ of order $p$,
we have by Herzog's theorem that
$$
  T\not\cong D_8 \text{ and } \exists \;C_p\times C_p\normal T \Leftrightarrow
  \#\{g\in T \backslash Z(T)|\; g^p=1\} \equiv 0\pmod{p^2}.
$$
For any given line in $Z(T)$, the number of $C_p\times C_p \le T$ intersecting
$Z(T)$ in that line is therefore divisible by $p$ by the same counting as
in case 1, and so the number of normal $C_p\times C_p$ in $T$ that intersect
$Z(T)$ in a line is divisible by~$p$, as required.
\end{proof}

\begin{proposition}\label{prop:kernelprank1}
Suppose that $G$ has a primitive relation. Then
either $K=\{1\}$ or $K\iso D_8$ or $K$ has normal $p$-rank one. In particular,
$K$ has cyclic centre.
\end{proposition}

\begin{proof}
If $K$ is not of these three types, then by Lemma \ref{lem:herzog},
the set of normal $C_p\times C_p$ in $K$ has cardinality coprime to $p$.
The $p$-group $P$ acts on this set by conjugation, so there is a fixed point.
In other words, there is $N=C_p\times C_p \normal K$ that is fixed under
conjugation by $P$, so $N\normal P$.
Now, either $N$ is in the centre of $P$, in which case it is also
in the centre of $G$ (since $K$ commutes with $C$ by definition),
and $G$ has no primitive relations
by Lemma \ref{lem:noncycliccentre}; or $N$ is a normal non-central
subgroup of $P$ that intersects $K$ non-trivially, and then $G$ has no
primitive relations by Lemma \ref{lem:normalnoncentral}.
\end{proof}

\begin{lemma}
\label{lemcl2}
If $C_{l^2}\le C$ for some prime $l$, then $\Prim(G)=0$.
\end{lemma}

\begin{proof}
Write $C=C_{l^n}\times \tilde{C}$ with $\tilde{C}$ cyclic of order
prime to $l$.
There is a unique $C_l\normal C$,
and any subgroup of $G$ that does not contain it is contained in 
$\tilde{C}\rtimes P$ and, a fortiori,
in $D=(C_l\times \tilde{C})\rtimes P\lneq G$.
Since $C_l \normal G$, we are done by Corollary \ref{cor:smallgpbiggp}.
\end{proof}

\begin{assumption}\label{not:quasielem}
In view of \ref{prop:kernelprank1} and \ref{lemcl2}, from now we assume:
\begin{enumerate}
\item $G=C\rtimes P$, with $P$ a $p$-group, and
$C=C_{l_1}\times \ldots\times C_{l_t}$ cyclic with $t$ distinct primes 
$l_j\ne p$.
\item
$K=\ker(P\to\Aut C)$ is either trivial, or isomorphic to $D_8$ or has normal $p$-rank one. 
\end{enumerate}
\end{assumption}

\begin{notation}\label{not:CK}
The following notation will be used in the rest of the section. Here, $N$
is any normal subgroup of $G$, and $j$ is an index, $1\le j\le t$.\\[-6pt]
\begin{center}
  \def\ffootnote#1{\addtocounter{footnote}{1}%
    \insert\footins{\footnotesize{$^{\arabic{footnote}}$#1}}%
    \addtocounter{footnote}{-1}%
  }
  \ffootnote{  
    If $K\not\iso Q_8$ is non-trivial, then it contains a unique 
    cyclic subgroup of index $p$, which is normal in $G$. In $Q_8$, there 
    are three cyclic subgroups of index 2 and the 2-group $P$ acts\\\hskip -1.8em on
    them by conjugation, so this action has a fixed point, which is also 
    normal in $G$.}
\begin{tabular}{l|@{\ \ }l}
  $\centralCp$ & the unique central subgroup of $K$ (and of $G$) of order $p$,\\
   &when $K$ is non-trivial.\\
  $C_K$ & either $K$ if $K$ is cyclic, or a cyclic index 2 subgroup of $K$ that\\ 
  & is normal in $G$ otherwise%
  \footnote{\relax}.\\
  $\bigC$ & $CC_K$; this is the largest normal cyclic subgroup of $G$.\\
  $\cH_N$ & a set of representatives of conjugacy classes of subgroups of
  $G$\\ & that intersect $N$ trivially.\\
  $\cH^c_N$ & the set of subgroups of $G$ that intersect $N$ non-trivially.\\
  $\cH^c$ & short for $\cH_{\bigC}^c$.\\
  $\cH$ & short for $\cH_{\bigC}$; we take $\cH$ to
  consist of subgroups of $P$.\\
  $\cH_m$ & the set of elements of $\cH$ of maximal size.\\
  $C^j$ & $C_{l_1}\times\ldots\times\widehat{C_{l_j}}\times\ldots\times C_{l_t}$,
  the $l_j'$-Hall subgroup of $C$.\\
  $K_j$ & $\ker(P\rightarrow \Aut(C^j))=
  \bigcap_{i\neq j}\ker(P\rightarrow \Aut C_{l_i})$.\\
  & Thus $K\le K_j$ and $K_j/K$ is cyclic, as it injects into $\Aut C_{l_j}$.\\
  $\tilde{K}_j$ & $K_j\cap \ker(P\rightarrow \Aut C_K)$.\\[3pt]
\end{tabular}
\end{center}

\vskip 3pt

For elements $\Theta_1=\sum_H n_H H$ and $\Theta_2=\sum_H m_H H$ of the
Burnside ring of $G$, write
$$
\Theta_1\equiv \Theta_2\pmod{\cH_N^c}$$
if $n_H=m_H$ for all $H\in \cH_N$.
\end{notation}

Note that $\centralCp$, $C_K$, $\bigC$, $C^j$, $K_j$ 
are all normal (even characteristic) in~$G$, 
and $\bigC$ is the largest normal cyclic subgroup of $G$.
The quotient $\bar P=G/\bigC$ acts faithfully on $\bigC$ by conjugation
(as seen from the presentation of generalised quaternion, semi-dihedral 
and dihedral groups in Proposition \ref{prop:prankone}), and is therefore abelian.
In particular, $G$ is an extension
$$
  1\rightarrow \bigC\rightarrow G\rightarrow \smallP \rightarrow 1,
$$
of an abelian $p$-group by a cyclic group. 
Also, all $H\in \cH$ are abelian, as they inject into $G/\bigC\cong \smallP$.
Finally, $C_K\le \tilde{K_j}$, and the quotient 
$\tilde{K_j}/C_K\injects \Aut C_{l_j}$ is cyclic and acts trivially on $C_K$ 
by conjugation. It follows that every $\tilde{K_j}$ is abelian.

Any relation in which every term contains a non-trivial subgroup of $\bigC$
is imprimitive by Lemma \ref{lem:deflations}. So, to find
generators of $\Prim(G)$, we will from now on focus our attention on relations
that contain subgroups of $P$
not containing $\centralCp$, or, equivalently,
subgroups $H\in \cH$.

\subsection{Some Brauer relations}\label{sec:somerels}

In this subsection, we define several relations, which will later be shown
to generate $\Prim(G)$.


\begin{lemma}\label{lem:faithfulchar}
Let $H\in\cH$ and let $\phi$ be a faithful irreducible character of $\bigC$.
Then $\Ind^G_{\bigC}\phi$ is irreducible, and any irreducible character
of $G$ whose restriction to $\bigC$ is faithful is of this form. Moreover,
\[
  \langle \Ind_H^G\triv, \Ind^G_{\bigC}\phi \rangle = \frac{|\smallP|}{|H|}.
\]
\end{lemma}
\begin{proof}
Since $\smallP=G/\bigC$ acts faithfully on $\bigC$, it also acts
faithfully on the faithful characters of $\bigC$. By Mackey's formula,
\[
\langle \Ind^G_{\bigC}\phi,\Ind^G_{\bigC}\phi\rangle=
\langle\phi, \Res^G_{\bigC}\Ind^G_{\bigC}\phi\rangle = \sum_{g\in \bigC\backslash G/\bigC}
\langle\phi, ^g\phi\rangle = 1,
\]
i.e. $\Ind^G_{\bigC}\phi$ is irreducible. Moreover, if $\chi$ is any irreducible
character of $G$ whose restriction to $\bigC$ is faithful, then
by Clifford theory, all irreducible summands of $\Res^G_{\bigC}\chi$ lie in one
orbit under the action of $G$. Since any normal subgroup of $\bigC$ is
characteristic, all $G$-conjugate irreducible characters of $\bigC$
have the same kernel, so all irreducible summands of $\Res^G_{\bigC}\chi$ are
faithful. Thus, if $\phi$ is one of them, then $\chi=\Ind^G_{\bigC}\phi$ by
Frobenius reciprocity and by the first part of the lemma.

The rest of the lemma now follows by Mackey's formula:
\beq
   \langle \Ind_H^G\triv, \Ind_{\bigC}^G\phi \rangle &=&
   \langle \triv, \Res_H\Ind_{\bigC}^G\phi\rangle \\[1pt]
   &=&
   \langle \triv, \sum\limits_{x\in H\backslash G/\bigC}
     \Ind_{{}^x\bigC\cap H}^H \Res_{{}^x\bigC\cap H}{}^x\phi\rangle\cr
   &=& \sum\limits_{x\in H\backslash G/\bigC}
     \langle \triv,\Ind_{\{1\}}^H \Res_{\{1\}}{}^x\phi\rangle \cr
   &=& \sum\limits_{x\in H\backslash G/\bigC}
     \langle \triv,\Ind_{\{1\}}^H \triv\rangle
=   |H\backslash G/\bigC| = \displaystyle\frac{|\smallP|}{|H|}.
\eeq
\end{proof}

\begin{lemma}\label{lem:Clifford}
Let $G$ be any finite group, $N\normal G$ a normal subgroup, and 
$\Theta_0=\sum_{H\in\cH_N} n_H H\in B(G)$. For an element $\Lambda$ of $B(G)$ 
write $\tilde\Lambda\in R_\Q(G)$ for the associated virtual representation.
\begin{enumerate}
\item\label{item:Clifford} For any irreducible character $\phi$ of $G$,
$$
\langle \Ind_N^G\triv,\phi\rangle = \bigleftchoice{\dim\phi}{N\le \ker\phi}{0}{\text{otherwise}}.
$$
\item\label{item:HvsHN} If $\phi$ is an irreducible character of $G$ satisfying $N\le \ker\phi$,
then for every subgroup $H\le G$,
$$
\langle \Ind_H^G\triv,\phi\rangle = \langle \Ind^G_{HN}\triv,\phi\rangle.
$$
In particular, $\langle\phi, \tilde\Lambda-\widetilde{N\Lambda}\rangle=0$
for every $\Lambda\in B(G)$.
\item\label{item:deflations} Let $N_1,\ldots,N_r$ be a collection of
normal subgroups of $G$, and set
$\Theta_i = \Theta_{i-1} \!-\! N_i\Theta_{i-1}$ for $i=1,\ldots,r$.
If $\phi$ is an irreducible character of $G$ whose kernel contains some $N_i$,
then $\langle\tilde\Theta_r,\phi\rangle=0$.
\item\label{item:Ncyclic}
Suppose that $N$ is cyclic.
If $\Theta$ is a relation and $\Theta\equiv\Theta_0\pmod{\cH_N^c}$, then
$\langle\tilde\Theta_0,\phi\rangle=0$
for every
irreducible character $\phi$ of $G$ whose restriction to $N$ is faithful.

\item\label{item:Thetar} Suppose that $N$ is cyclic.
Let $N_1,\ldots,N_r$ and $\Theta_1,\ldots,\Theta_r$ be as in part
(\ref{item:deflations}),
and assume in addition that all $N_i$ are contained in $N$, and that
any normal subgroup of $G$ that
intersects $N$ non-trivially contains some $N_i$. Then
$\langle\tilde\Theta_r,\phi\rangle = \langle \tilde\Theta_0,\phi\rangle$
for every irreducible character $\phi$ of $G$ that is faithful on $N$.
In particular, $\Theta_r$ is a relation if and only if
$\langle \tilde\Theta_0,\phi\rangle=0$ for every such character.
\end{enumerate}
\end{lemma}
\begin{proof}
We implicitly rely on Frobenius reciprocity throughout the proof.
\begin{enumerate}
\item By Clifford theory, $\Res^G_N\phi$ is a sum of irreducible characters
of $N$
that all lie in one $G$-orbit. The claim follows form the fact that the trivial
character is a $G$-orbit in itself.
\item The assumptions imply that the $H$-invariants of the underlying vector
space of $\phi$ is the same as the $HN$-invariants, since the entire vector
space is $N$-invariant.
\item The operators $\Lambda\mapsto \Lambda-N_i\Lambda$ on $B(G)$
commute pairwise.
So $\Theta_r$ is of the form $\Theta-N_i\Theta$ for some
$\Theta\in B(G)$, and the claim follows from part (\ref{item:HvsHN}).

\item Let $\phi$ be faithful on $N$, and let $U\in \cH_N^c$.
Since $N$ is normal in $G$ and cyclic, $U\cap N\neq\{1\}$ is normal in $G$,
so by part (\ref{item:Clifford}), $\langle\Ind_{U\cap N}^G\triv,\phi\rangle =0$.
Also, $\Ind_U^G\triv$ is a direct summand of $\Ind_{U\cap N}^G\triv$,
so $\langle\Ind_{U}^G\triv,\phi\rangle =0$.
It follows that $\langle\tilde\Theta_0,\phi\rangle=
\langle\tilde{\Theta},\phi\rangle =0$.

\item
Suppose $\phi$ is faithful on $N$, and hence on each $N_i$.
Then for any $H\leq G$, $\Ind_{HN_i}^G\triv$ is a direct
summand of $\Ind_{N_i}^G\triv$, and $\langle\Ind_{N_i}^G\triv,\phi\rangle=0$
by part (\ref{item:Clifford}). We deduce that
$\langle\Ind_{HN_i}^G\triv,\phi\rangle=0$, and therefore
$\langle\tilde\Theta_r,\phi\rangle =\langle\tilde\Theta_0,\phi\rangle$,
as claimed.
For the last claim, if $\phi$ is not faithful on $N$ then by assumption,
$\ker\phi$ contains some $N_i$,
and the assertion follows from part (\ref{item:deflations}).
%
\end{enumerate}
\end{proof}

\begin{corollary}\label{cor:relmodC}
Let $H_i\in \cH$ and $\Theta_0 = \sum n_iH_i \in B(G)$.
For $1\le j\le t$ set $\Theta_j = \Theta_{j-1} - C_{l_j}\Theta_{j-1}$,
and set $\Theta_{t+1} = \Theta_t$ if $K$ is trivial and
$\Theta_{t+1} = \Theta_t-\centralCp\Theta_t$ otherwise. In other words,
$$
  \Theta_{t+1} = \sum_i n_i\sum_{U\le \bigC} \mu(|U|)H_iU,
$$
where $\mu$ denotes the Moebius function, and $U$ runs over all
subgroups of $\bigC$.
Then the following are equivalent:
\begin{enumerate}
\item $\Theta_{t+1}$ is a relation.\label{item:thetat}
\item $\sum\frac{n_i}{|H_i|}=0$.\label{item:nis}
\item There exists a relation $\Theta$ such that 
$\Theta\equiv \Theta_0\pmod{\cH^c}$.\label{item:theta}
\end{enumerate}
\end{corollary}
\begin{proof}
For an element $\Lambda$ of $B(G)$, denote its image in $R_\Q(G)$ by $\tilde\Lambda$.

By Lemma \ref{lem:Clifford} (\ref{item:Thetar}), part (\ref{item:thetat}) is
equivalent to the statement that
$\langle\tilde\Theta_{t+1},\phi\rangle=0$ for all irreducible characters
$\phi$ of $G$ that are faithful on $\bigC$. So the equivalence with
(\ref{item:nis}) follows from Lemma \ref{lem:faithfulchar}.

The equivalence of (\ref{item:thetat}) and (\ref{item:theta}) follows from
Lemma \ref{lem:Clifford} (\ref{item:Ncyclic}) and (\ref{item:Thetar}):
indeed, if there exists a relation $\Theta\equiv \Theta_0\pmod{\cH^c}$, then
by Lemma \ref{lem:Clifford} (\ref{item:Ncyclic}),
$\langle\tilde{\Theta}_0,\phi\rangle=0$ for all irreducible characters
$\phi$ of $G$ whose restriction to $\bigC$ is faithful. But then Lemma
\ref{lem:Clifford} (\ref{item:Thetar}) implies that $\Theta_{t+1}$ is a relation.
%
\end{proof}


\begin{corollary}\label{cor:relmodC2}
Let $H_1,H_2\in \cH$.
\begin{enumerate}
\item 
If $|H_1|=|H_2|$, then there is a relation $\Theta\equiv H_1-H_2\pmod{\cH^c}$.
\item
If $|H_2|=p|H_1|$, then there is a relation $\Theta\equiv H_1-pH_2\pmod{\cH^c}$.
\end{enumerate}
\end{corollary}

\begin{theorem}\label{thm:sameInters}
Fix an index $1\le j\le t$. 
For subgroups $H_1, H_2\in\cH$ of the same size, 
the following are equivalent:
\begin{enumerate}
\item
There exists a relation $\Theta\equiv H_1-H_2\pmod{\cH^c}$ that is induced
from $C^j\rtimes P$.
\item The element
\[
\sum_{U\le \bigC\atop C_{l_j}\not\le U}\mu(|U|)(H_1U-H_2U)
\]
of $B(G)$ is a relation.
\item $\Res^P_{\tilde{K_j}}(H_1-H_2-\centralCp H_1 + \centralCp H_2)$ 
is a relation.
\item
There exists an element $g\in P$ such that the
intersections $H_1\cap \tilde{K_j}$ and ${}^gH_2\cap \tilde{K_j}$ are contained 
in one another, and
$$
  [N_P(H_1\cap \tilde K_j): H_1\tilde K_j] = [N_P(H_2\cap \tilde K_j): H_2\tilde K_j].
$$
\end{enumerate}
\end{theorem}
\begin{proof}
By Lemma \ref{lem:Clifford} (applied to $G=C^jP, N=C^jC_K, \Lambda=H_1-H_2$),
the statements (1) and (2) are equivalent, and both equivalent
to the condition that
$$
\langle \Ind^G_{H_1}\triv,\chi\rangle= \langle \Ind^G_{H_2}\triv,\chi\rangle
$$
for all irreducible characters $\chi$ of $G$ whose restriction to
$C^jC_K$, equivalently to $C^j\centralCp$, is faithful. 
By Lemma \ref{lem:faithfulchar}, this is
automatically satisfied for those
$\chi$ whose restriction to $\bigC$ is faithful.
Let $\chi$ be an irreducible character of $G$
whose restriction to $\bigC$ has kernel $C_{l_j}$. 
Then by \cite[Theorem 6.11]{Isa},
$$
  \chi = \Ind_{C\tilde{K_j}}^G \rho
$$ 
for some $\rho$; here $C\tilde{K_j}$ is the stabiliser of a constituent of
$\Res_{\bigC}\chi$. Moreover,
$\Res_{\tilde{K_j}}\rho$ is irreducible, faithful on $\centralCp$, and
any irreducible character of $\tilde{K_j}$ that is faithful on $\centralCp$
is of the form $\Res_{\tilde{K_j}}\rho$ for some such $\rho$.
For $H=H_1$ or $H_2$,~we~have
\begin{eqnarray*}
  \langle \chi,\Ind^G_{H}\triv\rangle \!\!\!
  & =\!\!\! & \displaystyle \langle \Ind_{C\rtimes \tilde{K_j}}^G \rho,\Ind^G_{H}\triv\rangle 
   =  \displaystyle \!\!\!\sum_{\scriptscriptstyle C\tilde{K_j}\backslash G/{H}} \langle\rho,\Ind_{C\tilde{K_j}\cap {}^g{H}}^{C\tilde{K_j}}\triv\rangle
  \\
  & =\!\!\! & \displaystyle \!\!\!\sum_{\scriptscriptstyle C\tilde{K_j}\backslash G/{H}}
        \langle\rho,\Ind_{\tilde{K_j}}^{C\tilde{K_j}}\Ind_{\tilde{K_j}\cap {}^g{H}}^{\tilde{K_j}}\triv\rangle
   =  \displaystyle \!\!\!\sum_{\scriptscriptstyle C\tilde{K_j}\backslash G/{H}}
        \langle\Res_{\tilde{K_j}}\rho,\Ind_{\tilde{K_j}\cap {}^g{H}}^{\tilde{K_j}}\triv\rangle
  \\
  & =\!\!\! & \!\!\!\sum_{\scriptscriptstyle \tilde{K_j}\backslash P/{H}}
  \langle\Res_{\tilde{K_j}}\rho,\Ind^{\tilde{K_j}}_{\tilde{K_j}\cap {}^g{H}}\triv\rangle
   =  \langle \Res_{\tilde{K_j}}\rho,\Res^P_{\tilde{K_j}}\Ind_{H}^P\triv\rangle.
\end{eqnarray*}

So (1) and (2) are equivalent to the statement that for
any irreducible character $\phi$ of $\tilde{K_j}$ that is faithful on $\centralCp$,
$\langle \phi,\Res^P_{\tilde{K_j}}\Ind_{H_1}^P\triv\rangle=
\langle \phi,\Res^P_{\tilde{K_j}}\Ind_{H_2}^P\triv\rangle$.
This in turn is equivalent to (3), again by Lemma \ref{lem:Clifford}.

We now prove the equivalence of (1)-(3) to (4).
Let $\phi$ be an irreducible character of $\tilde{K_j}$, 
faithful on $\centralCp$. Its kernel, say $N$, is then necessarily cyclic. 
For $H=H_1$ or $H_2$, by Lemma \ref{lem:Clifford} (\ref{item:Clifford}),
$$
\langle \phi,\Res^P_{\tilde{K_j}}\Ind_H^P\triv\rangle =
\langle \phi,\!\!\sum_{\scriptscriptstyle \tilde{K_j}\backslash P/H}\!\!\Ind_{{}^gH\cap
\tilde{K_j}}^{\tilde{K_j}}\triv\rangle
 =  \#\{g \in P/H\tilde{K_j}\; |\; {}^gH\cap \tilde{K_j} \le N\}.
$$
If ${}^gH\cap \tilde{K_j}\lneq N$ for all $g\in P$, this is 0. Otherwise, 
replace $H$ by some ${}^g{H}$ such that ${}^gH\cap \tilde{K_j}\le N$
(this does not change $\langle \phi,\Res^P_{\tilde{K_j}}\Ind_H^P\triv\rangle$).
We find
\begin{equation}
\label{eq:innerprod2}
\begin{array}{llllll}
\langle \phi,\Res^P_{\tilde{K_j}}\Ind_H^P\triv\rangle &=&
 \#\{g \in P/H\tilde{K_j}\; |\; g \in N_P(H\cap \tilde{K_j})\}\cr
& = & [N_P(H\cap \tilde{K_j}):H\tilde{K_j}].
\end{array}
\end{equation}
This uses the fact that $H\tilde{K_j}$ is contained in
$N_P(H\cap \tilde{K_j})$, since $\tilde{K_j}$ is abelian and therefore
normalises its subgroups, and since it is normal in $P$, so that
$H\cap \tilde{K_j}$ is normal in $H$.

To deduce that (3) implies (4) (or rather the contrapositive),
assume without loss of
generality that $|{H_1}\cap \tilde{K_j}|\geq |{H_2}\cap \tilde{K_j}|$.
Suppose first that no conjugate of
${H_2}\cap \tilde{K_j}$ is contained in ${H_1}\cap \tilde{K_j}$. Saturate
${H_1}\cap \tilde{K_j}$ to a cyclic subgroup $N$ of $\tilde{K_j}$ with
$\tilde{K_j}/N$ cyclic. Then no conjugate of ${H_2}\cap \tilde{K_j}$~is
contained in $N$, so if $\phi$ is an irreducible character of $\tilde{K_j}$
with kernel $N$, then
$$
\langle \phi,\Res^P_{\tilde{K_j}}\Ind_{{H_2}}^P\triv\rangle
=0 \neq \langle \phi,\Res^P_{\tilde{K_j}}\Ind_{H_1}^P\triv\rangle.
$$
If instead 
$[N_P(H_1\cap \tilde K_j): H_1\tilde K_j] \ne [N_P(H_2\cap \tilde K_j): H_2\tilde K_j]$, 
then (\ref{eq:innerprod2}) shows that these inner products are 
not equal whenever they are non-zero.

Conversely, if (4) is satisfied, then above calculation yields
$$
  \langle \phi,\Res^P_{\tilde{K_j}}\Ind_{H_1}^P\triv\rangle
  = \langle \phi,\Res^P_{\tilde{K_j}}\Ind_{H_2}^P\triv\rangle
$$
for all irreducible characters $\phi$ of $\tilde{K_j}$ that are faithful on
$\centralCp$.
\end{proof}

\begin{proposition}\label{prop:cyclicKj}
Suppose that some $K_{j_0}$ is cyclic and let
$\Theta=\sum_{H\le G}n_HH$ be a relation with 
$n_H=0$ for all $H$ that contain $C_{l_{j_0}}\!$. Then
$$
  \sum_{\substack{|H\cap K_{j_0}|\le p^i\\HK_{j_0}=P}} n_H\equiv 0\pmod p
  \qquad \qquad \text{$\forall\, i\ge 0$}.
$$
\begin{proof}
Let $\Ind_{C\rtimes K_{j_0}}^G(\chi\otimes\varphi)$ be an irreducible
character of $G$, where $\chi$ is a one-dimensional
character of $C$ with kernel
$C_{l_{j_0}}$, extended to $C\rtimes K_{j_0}$ as in Remark
\ref{rmrk:charSemiDir}, and $\varphi$ is an irreducible character of $K_{j_0}$.
If $H\le G$ intersects $C^{j_0}$ non-trivially, then by Lemma \ref{lem:Clifford}
(\ref{item:Clifford}),
$$
\langle\Ind_{C\rtimes K_{j_0}}^G(\chi\otimes\varphi),\Ind_H^G\triv\rangle = 0,
$$
while for any $H\le P$,
\beq
\langle\Ind_H^G\triv,\Ind_{C\rtimes K_{j_0}}^G(\chi\otimes\varphi)\rangle & = &
\sum_{g\in K_{j_0}\backslash P/H}\langle\triv_{K_{j_0}\cap {}^gH},
\Res^{K_{j_0}}_{K_{j_0}\cap {}^gH}\varphi\rangle\\
& = & \sum_{g\in K_{j_0}\backslash P/H}\langle\triv_{{}^g(K_{j_0}\cap H)},
\Res^{K_{j_0}}_{{}^g(K_{j_0}\cap H)}\varphi\rangle\\
& = &
\#(K_{j_0}\backslash P/H)\cdot\langle\triv_{K_{j_0}\cap H},\Res^{K_{j_0}}_{K_{j_0}\cap H}\varphi\rangle.
\eeq
The last two equalities follow from the facts that
\begin{enumerate}
\item $K_{j_0}\normal P$,
so that ${}^g(K_{j_0}\cap H)=K_{j_0}\cap {}^gH$ is a subgroup of $K_{j_0}$,
\item $K_{j_0}$ is cyclic, so that ${}^g(K_{j_0}\cap H) = K_{j_0}\cap H$,
since both are subgroups of $K_{j_0}$ of the same order.
\end{enumerate}
So by assumption on $\Theta$, we must have
\begin{eqnarray}\label{eq:innerprod}
\sum_{H\le P}n_H\#(K_{j_0}\backslash P/H)\cdot
\langle\triv_{H\cap K_{j_0}},\Res^{K_{j_0}}_{H\cap K_{j_0}}\varphi\rangle=0
\end{eqnarray}
for any 1-dimensional character $\varphi$ of $K_{j_0}$. By Lemma
\ref{lem:Clifford} (\ref{item:Clifford}),
$$
\langle\triv_{H\cap K_{j_0}},\Res^{K_{j_0}}_{H\cap K_{j_0}}\varphi\rangle =
\leftchoice{1}{H\cap K_{j_0}\le\ker\varphi}{0}{\text{otherwise}}.
$$
Also, $\#(K_{j_0}\backslash P/H)$ is a power of $p$, since $K_{j_0}$ is normal
in $P$, and it is equal to 1 if and only if
$HK_{j_0}=P$. The result now follows by considering equation (\ref{eq:innerprod})
modulo $p$ for $\varphi$ with increasing kernels.
\end{proof}
\end{proposition}

\begin{proposition}\label{prop:Ptoosmall}
The following conditions are equivalent:
\begin{enumerate}
\item $P/K$ is generated by exactly $t$ elements;
\item $K_j\supsetneq K$ for $1\le j\le t$;
\item $P/K$ acts faithfully on $C$ but does not act faithfully on
any maximal proper subgroup of $C$.
\end{enumerate}
Moreover, if $G$ does not satisfy these conditions, then $\Prim(G)=0$.
\end{proposition}
\begin{proof}
The equivalence of (2) and (3) is clear. Suppose that for some $j$, $K_j=K$.
Then $P/K=P/K_j$ injects into $\Aut(C^j)$, which has rank $t-1$, so $P/K$ is
generated by less than $t$ elements. Conversely, if $K_j\gneq K$ for all~$j$, 
then any set of elements $\{g_r\}$, $g_r\in K_{j_r}\backslash K$,
generates a group of rank $t$ in $\Aut(C)$, since each $g_r$ acts non-trivially
on $C_{j_r}$ and trivially on $C^{j_r}$. So $P/K\le \Aut(C)$ cannot be
generated by less than $t$ elements.

Suppose that $G$ does not satisfy these conditions, let $j_0$ be such that
$K_{j_0} = K$, or equivalently that $P/K$ acts faithfully on $C^{j_0}$.
Then, $G_0=C^{j_0}\rtimes P$
satisfies Assumption \ref{not:quasielem}, so Corollary \ref{cor:relmodC}
applies to both $G$ and $G_0$. Thus there
exists a $G$-relation $\Theta\equiv \sum n_iH_i\pmod{\cH^c}$ if and only if
there exists a $G_0$-relation $\Theta\equiv \sum n_iH_i\pmod{\cH^c_{\centralCp C^{j_0}}}$,
which can then be induced to an imprimitive $G$-relation. Here
$\centralCp C^{j_0}$ is considered as a normal subgroup of $G^0$.
So all occurrences of
$H\in \cH$ in any $G$-relation can be replaced by groups intersecting
$\bigC$ non-trivially using imprimitive relations, so $G$ has no primitive relations.
\end{proof}

\subsection{Primitive relations with trivial $K$}\label{sec:Ktriv}

As before, we have $G=C\rtimes P$, where $C$ is a cyclic group of order
$l_1\cdots l_t$ for distinct primes $l_i\ne p$, and $P$ is
a $p$-group. Assume throughout this subsection that $K=\{1\}$, that is
$P$ acts faithfully on $C$. In particular, $P$ is abelian and
its $p$-torsion is an elementary abelian $p$-group of rank at most $t$.

If $t=1$, then $\Prim(G)$ has been described in Proposition \ref{prop:G20},
so we assume for the rest of the subsection that $t>1$.
Define $\cM$ to be the set of all index $p$ subgroups of $P$.
For each $M\in \cM$, define the {\em signature} of $M$ to be the vector in
$\mathbb{F}_2^t$ whose $j$-th coordinate is 1 if $K_j\subseteq M$ and 0
otherwise.

\begin{proposition}\label{prop:condsKtriv}
The following properties of $G$ are equivalent:
\begin{enumerate}
\item All $K_j=\bigcap_{i\neq j}\ker(P\rightarrow \Aut C_{l_i})$ have the
same, non-trivial image 
in the Frattini quotient $P/\Phi(P)$ of $P$.
\item Each subgroup of $P$ of index $p$ contains either every $K_j$ or none,
and both cases occur.
In other words, the set of signatures of elements of $\cM$ is $\{(1,\ldots,1),
(0,\ldots,0)\}$. 
\end{enumerate}
\end{proposition}
\begin{proof}
Note that, as $P$ acts faithfully on $C$, some $K_j$ is not the whole of~$P$, 
and so some $M\in{\mathcal M}$ has a non-zero signature. 
A subgroup $K_j$ has trivial image in $P/\Phi(P)$ if and only if it is
contained in all maximal proper subgroups of $P$ if and only if the
signatures of all $M\in\cM$ have a 1 in the $j$-th coordinate. Moreover,
$K_{1}$, $K_{2}$, say, have different non-trivial images in the Frattini
quotient if and only if there are two hyperplanes in $P/\Phi(P)$ containing
one but not the other if and only if there are two subgroups in $\cM$
with signatures $(1,0,\ldots)$ and $(0,1,\ldots)$.
\end{proof}

\begin{theorem}\label{thm:Ktriv}
The group $G$ has a primitive relation if and only if $G$ satisfies the
equivalent conditions of Proposition \ref{prop:condsKtriv}. If it does,
then $\Prim(G)\iso C_p$ and is generated by the relation
$$
  \sum_{U\le C} \mu(|U|)(MU-M'U),
$$
where $M,M'\in \cM$ have signatures $(1,\ldots,1)$ and $(0,\ldots,0)$,
respectively.
\end{theorem}

The proof will proceed in several lemmata.

\begin{lemma}\label{lem:generators}
The group $\Prim(G)$ is generated by relations
of the form $\Theta \equiv M-M'\pmod{\cH_C^c}$, for $M,M'\in \cM$.
\end{lemma}
\begin{proof}
If a relation contains no subgroup of $P$, then it is imprimitive by
Lemma \ref{lem:deflations}.
Let $\Theta = n_H H +\ldots$ be any relation
with $H\le P$ of index at least
$p^2$. Pick $M\in\cM$ that contains $H$. Filter $M$ by a chain of subgroups, each of
index $p$ in the previous, such that at each step, the image in some
$\Aut(C_{l_j})$ decreases. By Corollary \ref{cor:relmodC2}, we can
replace $H$ by a subgroup $H'$ in this chain and by subgroups intersecting $C$
non-trivially, adding the relation $\Theta_{t+1}$
from Corollary \ref{cor:relmodC}. Moreover, the added relation is induced from 
a subgroup (since $\langle H,H'\rangle\!\le\!M\!<\!P$), so the class in $\Prim(G)$
is unchanged. Next, we claim that each subgroup in the chain can
be replaced by (an integer multiple of) its supergroup in the chain and by
elements of $\cH_C^c$, using an imprimitive relation.
Let $H'\le H$ be an index $p$ subgroup 
such that $\Im(H\rightarrow \Aut C_{l_j})\neq 
\Im(H'\rightarrow \Aut C_{l_j})$ for some $j$. Then, the subgroup
$C_{l_j}\rtimes H/\ker(H\rightarrow\Aut C_{l_j})$ is a group of the form
discussed in Example \ref{exer:G20}, with $H'$ corresponding to $\tilde{H}$
in that example. Lifting the relation of that example from the quotient,
$H'$ can be replaced by $p\cdot H$, as claimed. So, in summary,
we can replace any $H\lneq P$ by elements of $\cM$ and subgroups
intersecting $C$ non-trivially, without changing the class in $\Prim(G)$.

Also, by Corollary \ref{cor:relmodC}, the coefficient of $P$ in any
relation is divisible by~$p$. So we can replace $P$ by a subgroup in 
${\mathcal M}$, again using the relation of Example \ref{exer:G20}, 
induced from the subquotient $C_{l_1}\rtimes P/\ker(P\rightarrow \Aut C_{l_1})$ 
(by Proposition \ref{prop:Ptoosmall}, we may assume that 
$\ker(P\rightarrow \Aut C_{l_1})\neq P$).

We have thus shown that we can replace any subgroup of $P$ by a subgroup
in $\cM$, without changing the class in $\Prim(G)$. Finally, by
using relations $\Theta\equiv M-M'\pmod {\cH^c_C}$, we can replace all subgroups
in $\cM$ by one of them. But the coefficient of this one must be zero
by Corollary \ref{cor:relmodC}, so the resulting relation is imprimitive.
Thus, $\Prim(G)$ is generated by relations $\Theta\equiv M-M'\pmod{\cH^c_C}$, as
claimed.
\end{proof}

\begin{lemma}
Let $\Theta$ be a relation of the form $\Theta\equiv M-M'\pmod{\cH_C^c}$
with $M,M'\in \cM$. Then its order in $\Prim(G)$ divides $p$.
\end{lemma}

\begin{proof}
Any occurrence of $pM$ in a relation can be
replaced by a proper subgroup of $M$ and groups intersecting $C$,
using the relation from
Example \ref{exer:G20}, and similarly for $M'$. Next, these strictly
smaller groups can all be replaced by one group of the same size,
as in the proof of Lemma \ref{lem:generators}, using imprimitive relations.
The resulting relation is $\equiv 0 \pmod{\cH_C^c}$
by Corollary \ref{cor:relmodC}, and so is imprimitive.
\end{proof}

\begin{lemma}\label{lem:commonEntry}
If $M,M'\in\cM$ have signatures that agree in some entry,
then there is an imprimitive relation $\Theta\equiv M-M'\pmod{\cH_C^c}$.
\end{lemma} 

\begin{proof}
Say the signatures agree in the $j$th entry.
If the common entry is~1,
then $M\cap K_j = M'\cap K_j = K_j$, and if it is 0, then the intersections
are both equal to the unique index $p$ subgroup of $K_j$. In either case,
there is an imprimitive relation of the required form by
Theorem \ref{thm:sameInters}.
\end{proof}

\begin{lemma}
If $M,M'\in\cM$ have opposite signatures both of which
contain 0 and~1, then there is an
imprimitive relation $\Theta\equiv M-M'\pmod{\cH_C^c}$.
\end{lemma}

\begin{proof}
Say the signatures of $M$ and $M'$ start $(0,1,\ldots)$
and $(1,0,\ldots)$, respectively. In particular, there exists
$g\in K_1\backslash M$ with $\langle M,g\rangle = P$ and $g^p\in M$
and $h\in K_2\backslash M'$ with $\langle M',h\rangle = P$ and $h^p\in M'$.
Since $M\cap M'$ is of index $p$ in $M$ and in $M'$, and since
$M'=\langle M\cap M',g\rangle$ and similarly for $M$, the group
$\langle M\cap M',gh\rangle$ is in $\cM$ and contains neither $K_1$ nor
$K_2$, i.e. it has signature $(0,0,\ldots)$. Thus we get the required
relation by applying the previous lemma twice.
\end{proof}

\begin{corollary}
If there exists $M\in \cM$ whose signature contains 0 and~1, then
$\Prim(G)$ is trivial. Otherwise, $\Prim(G)$ is generated by any
$\Theta\equiv M-M'\pmod{\cH_C^c}$ where $M$ and $M'$ have signatures
$(0,\ldots,0)$ and $(1,\ldots,1)$, respectively.
\end{corollary}

To conclude the proof of Theorem \ref{thm:Ktriv}, it remains to show:
\begin{lemma}
Suppose that no element of $\cM$ has a signature in which both 0 and 1
occur. Let $M,M'\in\cM$ have signatures
$(0,\ldots,0)$ and $(1,\ldots,1)$, respectively, let
$\Theta\equiv M-M'\pmod{\cH^c_C}$ be a relation. Then $\Theta$ is primitive.
\end{lemma}
\begin{proof}
Assume for a contradiction that $\Theta$ is a
sum of relations that are induced and/or lifted from proper subquotients.
Then, at least one summand must contain terms in $\cM$ with
signature $(0,\ldots,0)$ such that the sum of all coefficients of
these terms is not congruent to 0 modulo $p$. Moreover, by Corollary
\ref{cor:relmodC} (\ref{item:nis}), this relation must contain either
a term in $\cM$ with signature $(1,\ldots,1)$, or $P$. Since no
$M\in \cM$ contains a normal subgroup of $G$, this relation cannot be
lifted from a proper quotient, so it must be induced from a proper subgroup.
Since two distinct groups in $\cM$ generate $P$,
this proper subgroup must be of the form
$(C_{l_1}\times\ldots\times\widehat{C_{l_{j_0}}}\times\ldots\times C_{l_t})\rtimes P$.
By Proposition \ref{prop:cyclicKj}, applied with $p^i=|K_{j_0}|/p$,
the sum of the coefficients of $M\in \cM$ with signature $(0,\ldots,0)$
plus the sum of the coefficients of $H\le P$ that satisfy $HK_{j_0}=P$ and
$|H\cap K_{j_0}|\le |K_{j_0}|/p^2$ is divisible by $p$. By the same proposition,
applied with $p^i=|K_{j_0}|/p^2$, the second sum is divisible by $p$.
We deduce that the sum of the coefficients of
$M\in \cM$ with signature $(0,\ldots,0)$ is divisible by $p$, which is a
contradiction.
\end{proof}

\subsection{Primitive relations with non-trivial $K$}\label{sec:Knontriv}

Finally, we consider $G=C\rtimes P$, where $C$ is a cyclic
group of order
$l_1\cdots l_t$ for distinct primes $l_i$ different from $p$, $P$ is
a $p$-group and the kernel $K$ of $P\rightarrow \Aut C$ is non-trivial. 
By Proposition \ref{prop:kernelprank1}, if $G$ has a primitive
relation, then $K$ must be isomorphic to $D_8$ or
have normal $p$-rank one, so it is a group of
the type described in Proposition \ref{prop:prankone}. We will assume
this throughout this subsection. Note that
in particular, if $p$ is odd, then $K$ must be cyclic.

Recall that $\cH$ is a set of representatives of conjugacy classes of subgroups 
of $P$ that do not contain $\centralCp$ (the unique subgroup of $K$ of order $p$ 
that is central in $G$), and $\cH_m$ is the set of elements of $\cH$ of maximal size.

\begin{lemma}\label{lem:genKnontriv}
The group $\Prim(G)$ is generated by relations of the form
$$
  \Theta=\sum_{\tilde{C}\le \bigC} \mu(|\tilde{C}|)(\tilde{C}H_1-\tilde{C}H_2),
$$
for $H_1,H_2\in \cH_m$.
\end{lemma}
\begin{proof}
By Corollary \ref{cor:relmodC}, these are indeed $G$-relations.

Now let $\Theta$ be any relation. 
If no elements of $\cH$ occur in it,
then $\Theta$ is imprimitive by Lemma \ref{lem:deflations}.
Suppose $H\in\cH\setminus \cH_m$ occurs in $\Theta$ and
$H'\in\cH$ is such that $|H'|=p|H|$. Set
$I_H=\Im(H\rightarrow \Aut C)$, and consider two cases:

\smallskip

(1) $I_H = \Im(P\rightarrow \Aut C)$. Then the assumption that there
exists an element of $\cH$ of size $p|H|$ implies that $p=2$, $K$ is dihedral
or semidihedral, and $P$ is a direct product of $H$ and $K$. In particular,
$H$ is normal in $P$. Let $C_2^a$ be a non-central $C_2$ in $K$.
By inducing the relation
of Example \ref{ex:cpcp} from the subquotient
$HC_2^a\centralCp/H\cong C_2\times C_2$, we may replace $H$ by strictly bigger
groups in $\cH$ and by subgroups of $P$ containing $\centralCp$,
without changing the class of $\Theta$ in $\Prim(G)$.

(2)
$I_H\neq\Im(P\rightarrow \Aut C)$. Let $B$ be
an index $p$ subgroup of $\Im(P\rightarrow \Aut C)$ containing $I_H$.
By intersecting $H'$ with the pre-image of $B$ in $P$ if necessary,
we can find an index $p$ subgroup $H''$ in $H'$ such that $H$ and $H''$
generate a proper subgroup of $P$. Thus, the relation
$\sum_{\tilde{C}\le \bigC} \mu(|\tilde{C}|)(\tilde{C}H - \tilde{C}H'')$
of Corollary \ref{cor:relmodC2} is induced from a proper subgroup
by Proposition \ref{prop:biggroup}, so that $H$ can be replaced by
$H''$, which is a subgroup of $H'$,
and by subgroups that intersect $\bigC$ non-trivially without changing the
class in $\Prim(G)$. By inducing the relation of Example \ref{ex:cpcp} from
the subquotient $H'\centralCp/H''\cong C_p\times C_p$
(note that $H'$ is abelian and $\centralCp$ is central, so
$H''$ is indeed normal in the group generated by the two),
we may replace $H''$ by strictly bigger groups in $\cH$ and by
subgroups of $P$ containing $\centralCp$.

\smallskip

In summary, any class in
$\Prim(G)$ is represented by a relation of the form
$\Theta\equiv\sum n_iH_i\pmod{\cH^c}$ with $H_i\in\cH_m$.

Since $\sum_i n_i = 0$ by Corollary
\ref{cor:relmodC}, the generators of $\Prim(G)$ are as claimed.
\end{proof}

\begin{lemma}\label{lem:Fpvs}
The group $\Prim(G)$ is an elementary abelian $p$-group.
\end{lemma}
\begin{proof}
By the previous lemma, it suffices to show that for any $H_1,H_2\in \cH_m$,
$$
  \Theta=p\cdot\sum_{\tilde{C}\le \bigC} \mu(|\tilde{C}|)(\tilde{C}H_1-\tilde{C}H_2)
$$
is imprimitive. Let $A$ be a subgroup of $\Im(P\rightarrow \Aut C)$ of
index $p$ and such that for some $j$,
$A\cap \Aut C_{l_j}\neq \Im(P\rightarrow \Aut C_{l_j})$.
By intersecting $H_1$ and $H_2$ with
the pre-image of $A$ in $P$, we may find subgroups $H_3\le H_1$ and
$H_4\le H_2$ of index $p$ whose image in $\Aut C$ lies in $A$,
and in particular whose image in $\Aut C_{l_j}$ is strictly smaller than that
of $H_1$ and of $H_2$, respectively. By inducing the relation of Example
\ref{exer:G20}, we may replace $pH_1$ and $pH_2$ in $\Theta$ by $H_3$ and
$H_4$, respectively, and by groups containing $C_{l_j}$,
without changing the class of $\Theta$ in $\Prim(G)$.
Now, we can replace $H_3$ by $H_4$ and by groups intersecting
$\bigC$ non-trivially, using the relation of Corollary \ref{cor:relmodC2} (1).
Since $H_3$ and $H_4$ together generate a proper subgroup of $G$
(it is contained in the pre-image of $A$ in $P$), the class
of $\Theta$ in $\Prim(G)$ is still unchanged. But now, the only
element of $\cH$ appearing in $\Theta$ is $H_4$, so by Corollary
\ref{cor:relmodC}, it must appear with coefficient 0 and the resulting
relation is imprimitive by Lemma \ref{lem:deflations}.
\end{proof}

It only remains to determine the rank of $\Prim(G)$.
We will first treat separately the case that $K=\centralCp$ and
$K\normal P$ is a direct summand.
In this case, $K_j\iso K\times\text{(cyclic group)}$ 
for every $j$, and their images in $P/\Phi(P)$ are either 
$C_p$ or $C_p\times C_p$.

\begin{proposition}\label{prop:Pdirprod}
Suppose $P$ is a direct product by $K\iso\centralCp$. If some 
$K_j$ has image $C_p$ in $P/\Phi(P)$ or some
$K_{j_1}$, $K_{j_2}$ have different images in $P/\Phi(P)$, then
$\Prim(G)$ is trivial. Otherwise, $\Prim(G)\cong \F_p^{p-2}$.
\end{proposition}
\begin{proof}
Denote by $\cdot^\Phi$ the image of $\cdot$ in the Frattini quotient $P/\Phi(P)$.

Let $H=\langle a_1,\ldots,a_r\rangle$ be a complement to $K$ in $P$,
where $r$ is the smallest number of generators of $H$.
If $r$ is less than the number $t$ of prime divisors of $|C|$, then by
Proposition \ref{prop:Ptoosmall}, some $K_j$ is equal to $K\cong C_p$,
and so $K_j^\Phi\iso C_p$. Also, by the same proposition,
$G$ has no primitive relations, as claimed.

Suppose from now on that $r=t$.
Write $K=\langle c\rangle$. The elements of $\cH_m$ 
are precisely the complements of $K$ in $P$, so they are shifts of $H$ of 
the form $H_\delta=\langle c^{\delta_1}a_1,\ldots,c^{\delta_r}a_r\rangle$
for $\delta=(\delta_1,...,\delta_t)\in\F_p^t$.
\newline
\textbf{Step 1.} 
Suppose $K_j^\Phi\iso C_p$ for some $j$. Then $K_j^\Phi=K^\Phi$, so
the intersection of $K_j$ with any $H_\delta$
has trivial image in the Frattini quotient, and therefore consists only
of $p$-th powers. For any $H_\delta\in\cH_m$, from the explicit description 
of the generators it follows that $H_\delta\cap K_j=H\cap K_j$, 
since taking $p$-th powers kills $c$. So by Theorem \ref{thm:sameInters},
there exists an imprimitive relation $\Theta\equiv H-H_\delta\pmod{\cH^c}$.
Combined with Lemma \ref{lem:genKnontriv}, this implies that $\Prim(G)$
is generated by a relation of the form $\Theta\equiv n_HH\pmod{\cH^c}$. 
But $n_H=0$ by Corollary \ref{cor:relmodC}, and so $\Prim(G)$ is trivial
by Lemma \ref{lem:deflations}.
\newline
\textbf{Step 2.}
Suppose $K_{j_1}^\Phi \ne K_{j_2}^\Phi$, and both are two-dimensional.
Then, given any lines $L_1\le K_{j_1}^\Phi$, 
$L_2\le K_{j_2}^\Phi$ distinct from $K^\Phi$, 
we can lift a hyperplane in $P/\Phi(P)$ that intersect each
$K_{j_i}^\Phi$ in $L_i$ for $i=1,2$ to an index $p$-subgroup of $P$ that
intersects $K$ trivially. Thus, given any two complements
$H_1$ and $H_2$, we can find $H_3$ such that $H_i\cap K_{j_i}=
H_3\cap K_{j_i}$ for $i=1,2$. Thus, there exist imprimitive
relations $\Theta_i\equiv H_i-H_3\pmod{\cH^c}$ for $i=1,2$ and so,
$\Prim(G)$ is trivial by the same argument as in the previous step.
\newline
\textbf{Step 3.}
From now on, suppose that $K_1^\Phi=\ldots=K_t^\Phi \iso C_p\times C_p$. 
Denote the $p+1$ lines in this quotient by $K^\Phi,L_1,...,L_p$.
For any $H\in\cH_m$, the image $(H\cap K_j)^\Phi$ 
is one of the lines $L_i$. This line is the same for all $j$ 
(any two $L_{i_1}\ne L_{i_2}$ generate $K_1^\Phi$, forcing
$H$ to contain $K$ otherwise). Consider the linear map
$$
  l: K(G)\rightarrow\F_p^p,
$$
that takes $H\in\cH_m$ to the $i$th basis vector when $(H\cap K_j)^\Phi=L_i$,
and declaring $l(H)=0$ for $H\notin\cH_m$. 

We claim that every relation $\Theta\in\ker l$ is imprimitive. 
To this end, we first modify $\Theta$ to get rid of the subgroups that are 
in $\cH$ but not in $\cH_m$.
Fix $H\in \cH_m$ (a complement to $K$ in $P$), and let $H_1\in \cH\setminus \cH_m$. 
Denote $\Im(H_1\rightarrow \Aut(C))$ by $A$. 
Then $H_1\cong A \lneq \Im(P\rightarrow \Aut(C))$,
so intersecting $H$ with the preimage of $A$ in $P$, we obtain a
proper subgroup $H_2\le H$
of the same order as $H_1$, and such that $\langle H_1,H_2\rangle \neq P$,
since its image in $\Aut(C)$ is contained in $A$.
The relation $\Omega_1\equiv H_1-H_2\pmod{\cH^c}$ of Corollary
\ref{cor:relmodC2} (1) is
therefore imprimitive, and it clearly lies in the kernel of $l$. So by adding
relations of the type $\Omega_1$ to $\Theta$,
we may replace any terms in $\Theta$ that lie in $\cH$ but not in $\cH_m$
by subgroups of $H$, without changing the class in $\Prim(G)$ and without
changing $l(\Theta)$.
Next, given any $\tilde{H}\lneq H$,
we can obtain an imprimitive relation of the form $\Omega_2\equiv
\tilde{H} - pH\pmod{\cH^c}$
by inducing the relation of Corollary \ref{cor:relmodC2} (2) from the proper
subgroup $C\rtimes H$ of $G$. This relation also lies in the kernel of $l$.
So by adding relations of the type $\Omega_2$ to $\Theta$,
we may assume without loss of generality that
$\Theta\equiv \sum_{H\in \cH_m}n_HH\pmod{\cH^c}$.

Let $H_1,H_2\in \cH_m$ be such that their intersection with some $K_j$
has the same image in $P/\Phi(P)$. We claim that this implies that 
$H_1\cap K_j = H_2\cap K_j$. Indeed, $K_j$ is of the form
$K_j= \centralCp\times \langle g\rangle$ for some $g\in P$, and
each of these intersections is a complement of $\centralCp$ in $K_j$. There
are $p$ such complements, and they all have distinct images in $P/\Phi(P)$.
By Theorem \ref{thm:sameInters}, there is an imprimitive relation
$\Omega_3\equiv H_1-H_2\pmod{\cH^c}$, which is in the kernel of $l$. So we may
assume without loss of generality that $\Theta\equiv \sum_H n_HH\pmod{\cH^c}$,
where all $H$ in the sum give rise to different lines $L_k$ in $P/\Phi(P)$.
Since $\Theta$ is assumed to lie in the kernel of $l$, we have
$p|n_H$ for all $H$, and so $\Theta$ is imprimitive
by Lemma~\ref{lem:Fpvs}.
\newline
\textbf{Step 4.}
By Corollary \ref{cor:relmodC}, the image of $l$ is precisely equal
to the hyperplane $V=\{(v_1,\ldots,v_p)|\sum v_i = 0\in \F_p\}$ of $\F_p^p$.
By Step 3, $\Prim(G)$ is isomorphic to the quotient of $V$ by the image of
the imprimitive relations under $l$. Let $\Theta$ be imprimitive with
non-trivial image under $l$,
without loss of generality assume that it is primitive in its own
subquotient $\bar{G}$. Since $l(\Theta)\neq 0$ and since
any two groups in $\cH_m$ generate all of $P$,
this subquotient is either a quotient of $P$ or of the form $\tilde{C}\rtimes P$,
where $\tilde{C}$ is a proper subgroup of $C$. So, assume that 
$\Theta$ is as described by Theorem \ref{thm:pgroups} or by Lemma
\ref{lem:genKnontriv}.

Suppose first that it is the former. Then $\bar{G}=P/N$ must be
isomorphic to $C_p\times C_p$, since $P$ is abelian, and $\Theta$ is then
lifted from the relation of Example \ref{ex:cpcp}.
Since the projections of all $H\in \cH_m$ in this subquotient have
the same size, and since we assume that $0\neq l(\Theta) \in V$, we deduce
that the projections of $H\in \cH_m$ in this subquotient are cyclic of order
$p$. Moreover, any two distinct $H,H'\in \cH_m$ generate all of $P$, and neither
contains $\centralCp$, so the projection of $\centralCp$ onto this subquotient
is non-trivial cyclic, and the other $p$ terms in $\Theta$ of size $p$
correspond to distinct elements of $\cH_m$. If $g$ is a lift of a generator
of one such cyclic group from $\bar{G}$ to $P$, and if $c$ is a generator of
$\centralCp$, then the elements of $\cH_m$ entering in $\Theta$ are
$H_i=\langle c^ig\rangle N$, $i=0,\ldots, p-1$. Since the image of $c$ in the
Frattini quotient $P/\Phi(P)$ is non-trivial, it follows that the images of
$H_i\cap K_j$ in $P/\Phi(P)$ for $i=0,\ldots,p-1$ are either all the same or
all distinct. The assumption that $l(\Theta)\neq 0$ then forces the latter,
and so $l(\Theta) = (-1,-1,\ldots,-1)$.

Now suppose that $\bar{G}=\tilde{C}\rtimes P$, where $\tilde{C}$ is a proper
subgroup of $C$. Then by Theorem \ref{thm:sameInters},
$\Theta\equiv H_1-H_2\pmod{\cH_C^c}$ such that $H_1\cap K_j=H_2\cap K_j$
for some $j$. But then $l(\Theta)=0$, which is a contradiction.

In summary, the image of the
imprimitive relations under $l$ is a one-dimensional subspace of $V$, spanned by
$(1,\ldots,1)$,
so $\Prim(G)\cong \F_p^{p-2}$, as claimed.
\end{proof}

\begin{thm}\label{thm:Knontriv}
Assume that either $|K| > p$, or $P$ is not a direct product by $K$.
Let $\cH_m$ be the set of subgroups of $P$ of maximal size among those that
intersect the centre of $K$ trivially. Define
a graph $\Gamma$ whose vertices are the elements of $\cH_m$ and with an edge
between $H_1,H_2\in \cH_m$ if one of the following applies:
\begin{enumerate}
\item the subgroup generated by $H_1$ and $H_2$ is a proper subgroup of $P$;
\item $t>1$ and there exists $1\le j_0\le t$ such that
$H_1\cap \tilde{K}_{j_0}=H_2\cap \tilde{K}_{j_0}$, where
$\tilde{K}_{j_0}=K_{j_0}\cap
\ker(P\rightarrow \Aut C_K)$ (recall that $C_K$ is a fixed maximal
cyclic subgroup of $K$ that is normal in $G$, see Notation \ref{not:CK});
\item the intersection $H_1\cap H_2$ is of index $p$
in $H_1$ and in $H_2$, and $\langle H_1,H_2\rangle/H_1\cap H_2$ is either dihedral,
or the Heisenberg group of order $p^3$.
\end{enumerate}
Let $d$ be the number of connected components of $\Gamma$.
Then $\Prim(G)\!\iso\!(C_p)^{d-1}$, generated by relations
$\Theta=\sum_{\tilde{C}\le \bigC} \mu(|\tilde{C}|)(\tilde{C}H_1-\tilde{C}H_2)$
for $H_1,H_2\in\cH_m$ corresponding to
distinct connected components of the graph.
\end{thm}
\begin{proof}
The three conditions for when there is an edge between
$H_1$ and $H_2\in \cH_m$ ensure that if $H_1$ and $H_2$ lie in the same
connected component of the graph $\Gamma$, then there is an imprimitive
relation $\Theta\equiv H_1-H_2\pmod{\cH^c}$, by using Proposition
\ref{prop:biggroup}, Theorem \ref{thm:sameInters}, and by inducing the
relations of Theorem \ref{thm:pgroups}, respectively.

For a subgroup $H\in\cH_m$ write $[H]$ for the connected component of 
$\Gamma$ that contains $H$. Note that since the conjugation action of $P$
on its Frattini quotient is trivial, condition (1) ensures that $[H]=[{}^gH]$
for any $g\in G$. Therefore $[\,\cdot\,]$ extends by linearity to a well-defined
linear map $B(G)\to\F_p^d$, defining it to be 0 on the groups not in $\cH_m$.
We are interested in its restriction to the space of relations, 
$$
\begin{array}{cccc}
  [\,\cdot\,]:&K(G) &\lar& \F_p^d.\cr
\end{array}
$$
By Corollary \ref{cor:relmodC}, the image of this restriction is the hyperplane
$V=\{{\mathbf v}|\sum v_i=0\}$. We will show that this map establishes an
isomorphism between $V$ and $\Prim(G)$.

First, we claim that every imprimitive relation is in 
$\ker[\,\cdot\,]$, so that $[\,\cdot\,]$ yields a well-defined map
$$
  [\,\cdot\,]: \>\>\Prim G \lar \F_p^{d-1}.
$$
Suppose, on the contrary, that
$[\Theta]\ne 0$ and $\Theta$ is imprimitive.
So $\Theta=\sum_i\Theta_i$, where each $\Theta_i$ comes from a proper 
subquotient of $G$. 
Without loss of generality, we may assume that each of these 
summands is primitive in its subquotient. Moreover, using
Lemma \ref{lem:genKnontriv} and Theorem \ref{thm:Ktriv},
we may assume further that $\Theta_i$ that are
induced from $p$-groups are of the form
described in Theorem \ref{thm:pgroups}, while $\Theta_i$ that are
induced/lifted from quasi-elementary subquotients that are not $p$-groups
are as described by Proposition \ref{prop:G20} and by Lemma \ref{lem:genKnontriv}.

Because $[\Theta]\ne 0$, some $[\Theta_i]$ is not $0$.
The entries of $[\Theta_i]\in\F_p^d$ sum up to~0, so at least
two of them are non-zero. In particular, $\Theta_i$ contains 
two terms $H_1, H_2\in \cH_m$ from two different connected components of $\Gamma$,
appearing in $\Theta_i$ with non-zero coefficients modulo $p$.
Since both $H_i$ act faithfully on $\bigC$, their intersection does not
contain any normal subgroup of $G$, so $\Theta_i$ 
must be induced from a proper subgroup of $G$. Since $H_1$ and $H_2$ lie
in different connected components of $\Gamma$, they generate all of $P$.
So $\Theta_i$ is either induced from $P$ or from $\tilde{C}\rtimes P$ for a
proper non-trivial subgroup $\tilde{C}$ of $C$.

If $\Theta_i$ is induced from $P$, then it is induced from a subquotient
of the form described in Theorem \ref{thm:pgroups} and the images of 
$H_1, H_2$ are of order $p$ in it. In fact, since $\langle H_1,H_2\rangle = P$,
this subquotient is a quotient. If it is dihedral or a Heisenberg group
of order $p^3$, then there is an edge between
$H_1$ and $H_2$ - contradiction. Otherwise, it is isomorphic
to $C_p\times C_p$, so $|P| = p|H_1|$.
It follows that $K=\centralCp$ and $P=\centralCp\times H_1$,
and this case was excluded.

From now on we may assume that $\Theta_i$ is induced from a subgroup 
$\tilde{C}\rtimes P$. 
Let $\tilde{K} = \ker(P\rightarrow \Aut(\tilde{C}))$.
Since $H_1, H_2$ are abelian and generate $P$, 
their intersection is normal in $P$, and so is $I=\tilde{K}\cap H_1\cap H_2$. 
Since the image of $\centralCp$ in $\tilde{K}/I$ is non-trivial,
$\Theta_i$ cannot be the relation of Proposition \ref{prop:G20}, so it
must be as described by Lemma \ref{lem:genKnontriv}.
Moreover, since $\Theta_i$ is primitive in its subquotient,
Proposition \ref{prop:kernelprank1} implies that $\tilde{K}/I$
is isomorphic to $D_8$ or has normal $p$-rank one. 
Pick an index $j$ with $\tilde{C}\le C^j$ and $K_j\le \tilde{K}$;
see Notation \ref{not:CK}.
Then $K_j/K_j\cap I$ is canonically identified with a non-trivial normal
subgroup of $\tilde{K}/I$,
and hence is itself isomorphic to $D_8$ or has normal $p$-rank one,
or is isomorphic to $C_2\times C_2$, the latter being only possible if
$\tilde{K}/I\iso D_8$.

If $K_j/K_j\cap I$ is isomorphic to $D_8$ or has normal $p$-rank one, then
$\tilde{K}_j/\tilde{K}_j\cap I$ is cyclic, and so 
$\tilde{K}_j\cap H_1\cap H_2$ is a maximal (with respect to inclusion)
cyclic subgroup of $\tilde{K_j}$
not containing $\centralCp$. But $\tilde{K}_j\cap H_1\cap H_2\le \tilde{K_j}\cap H_1,
\tilde{K}_j\cap H_2$, which are also cyclic and do not contain $\centralCp$,
so necessarily $\tilde{K}_j\cap H_1 = \tilde{K}_j\cap H_2$, and there is
therefore an edge between $H_1$ and $H_2$.

Finally, suppose $K_j/K_j\cap I\cong C_2\times C_2$ and $\tilde{K}/I\iso D_8$.
By Proposition~\ref{prop:Ptoosmall}, these two assumptions and 
the inclusions $\{1\}\lneq K\lneq K_j\lneq \tilde K$ force the index 
$[C:\tilde C]$ to be the product of exactly two primes $l_j$, $l_i$.
In other words, $\tilde K=K=\centralCp\iso C_2$, and $\tilde{K}_j/\tilde{K}_j\cap I,
\tilde K_i/\tilde{K}_i\cap I$ are the two distinct subgroups of $D_8$ isomorphic to
$C_2\times C_2$ . 
The intersections $H_1\cap \tilde K$, $H_2\cap \tilde K$ meet $\centralCp$
trivially, so their images in the quotient $\tilde K/I$ are
either trivial or non-central of order 2. If these images are conjugate,
or if at least one of them is trivial, then either 
$H_1\cap \tilde K_j$ is conjugate to $H_2\cap \tilde K_j$ or 
$H_1\cap \tilde K_i$ is conjugate to $H_2\cap \tilde K_i$; in both cases,
there is an edge between $H_1$ and $H_2$. 
So suppose their images in $\tilde K/I\cong D_8$ are two non-conjugate
non-central subgroups of order 2.
Say, $H_1\cap \tilde K_j$ becomes isomorphic to $C_2$ in $\tilde K/I$, and 
$H_2\cap \tilde K_j$ becomes trivial. 
Then by \cite[Lemma 6.15]{Bouc}, applied to the
subgroup $E=K_i/K_i\cap I\cong C_2\times C_2$ of $P/I$, with $H=H_1/H_1\cap I$
there exists a subgroup $H'$ of $P/I$ that centralises $E$, and a relation 
$$
  \tilde\Omega = H - H' - H\centralCp + H'\centralCp
$$
in $P/I$. Lifting it to $P$, we get a relation
$$
  \Omega = H_1 - H_3 - H_1\centralCp + H_3\centralCp 
$$
for some $H_3\in P$. By Corollary \ref{cor:relmodC}, $H_3\in \cH_m$.
We already showed that the existence of such a relation forces
$H_1$ and $H_3$ to lie in the same connected component. However, 
$H_2\cap \tilde K_j=I=H_3\cap \tilde K_j$, since no non-central element of $D_8$
can lie in one $C_2\times C_2$ and centralise the other one, and
so there is an edge between $H_2$ and $H_3$.
So in this case $[H_1]=[H_2]$ as well.

Finally, to determine the kernel of
\begin{eqnarray}\label{eq:primG}
  [\,\cdot\,]: \>\>\Prim G \lar \F_p^{d-1},
\end{eqnarray}
it suffices to evaluate it on linear combinations of the generators of
$\Prim(G)$ given by Lemma \ref{lem:genKnontriv}. Such a linear
combination is mapped to 0 if and only if the coefficients of all $H\in \cH_m$
are divisible by $p$. We deduce by Lemma \ref{lem:Fpvs} that the map
(\ref{eq:primG}) is an isomorphism.
\end{proof}

\begin{remark}\label{ex:Cprime}
This completes the proof of Theorem \ref{thm:A} in the last
remaining case, when $G=C\rtimes P$ is quasi-elementary with $P$ a $p$-group
and $C$ cyclic of order prime to $p$.

The conditions in Theorem \ref{thm:A}(4) that
describe when such a $G$ has primitive relations
are group-theoretic, but they are rather intricate.
In the special case that $|C|=l\neq p$ is prime, they can be made
completely explicit, and one can list all such $G$
in terms of generators and relations.
We refer the interested reader to \cite{brauer2}, and just make one
remark here.

Suppose that $G$ has a primitive relation.
By Proposition \ref{prop:kernelprank1}, the kernel $K$ of the action of $P$
on $C$ by conjugation is $\{1\}$, $D_8$ or has normal $p$-rank one.
Suppose that $\{1\}\ne K\ne P$ (cf. Example \ref{exer:G20},
Lemma \ref{lem:StructOfP}).
Write $A$ for the image of $P$ in $\Aut C$.
What makes the case $|C|=l$ simpler is that in this case $A$ is cyclic and
the sequence
\begin{equation}
\label{KPAseq}
   1 \to K \to P \to A \to 1
\end{equation}
must split; this makes $P=K\rtimes A$ and $G$ not hard to describe by
generators and relations.

Indeed, suppose the sequence does not split.
If $K$ is cyclic or generalised quaternion, then all
subgroups of $K$ contain the central $\centralCp$, so, using the notation of
section \ref{sec:Knontriv}, $\cH$ consists of subgroups of $P$ that intersect
$K$ trivially. Since there is no subgroup $H$ of $P$ with $H\cap K=\{1\}$ and
surjecting onto $A$ (otherwise $P$
would be a semi-direct product of $H$ by $K$), all subgroups in $\cH$
must be contained in the pre-image under $P\rightarrow \Aut C$ of the
unique index $p$ subgroup. Thus, there is an edge between any two groups in
$\cH_m$, using the notation of Theorem \ref{thm:Knontriv}, and so
$\Prim(G)=\{1\}$. Now suppose $K$ is dihedral or semi-dihedral, and
denote by $C_K$ the unique cyclic index~2 subgroup of $K$.
Since the automorphism of $C_K$ given by any non-central involution of $K$
is not divisible, $P$ not being a semi-direct product by $K$ implies that
it is not a semi-direct product by $C_K$ either. Thus, again, there is an
index $p$ subgroup of $P$ containing
any subgroup of $P$ that does not intersect $C_K$, so the same argument
applies and shows that $\Prim(G)=\{1\}$.

Finally, let us mention that when $C$ has composite order,
it may happen that the sequence \eqref{KPAseq}
does not split, but $G$ still has primitive relations.
The smallest such example that we know is a group $G$ of order
$3934208=2^{11}\cdot 17\cdot 113$, with $C=C_{17}\times C_{113}$, $K=C_8$
and $A=C_{16}\times C_{16}$.\footnote{
In Magma, this group may be given by
PolycyclicGroup$\langle a,b,c,d,e,f,g,h,i,j,k,l,m|\\
  a^2\!=\!f,b^2\!=\!e,c^2\!=\!d,d^2\!=\!h,e^2\!=\!g,f^2\!=\!i,g^2\!=\!k,h^2\!=\!j,i^2,j^2,k^2,l^{17},m^{113},{}^ac\!=\!c\!\cdot\! i,{}^bc\!=\!c\!\cdot\! f,\\
  {}^bd\!=\!d\!\cdot\! i,{}^ce\!=\!e\!\cdot\! i,{}^bl\!=\!l^3,{}^el\!=\!l^9,{}^gl\!=\!l^{13},{}^kl\!=\!l^{16},{}^bm\!=\!m^{48},{}^cm\!=\!m^{42},{}^dm\!=\!m^{69},{}^em\!=\!m^{44},\\
  {}^gm\!=\!m^{15},{}^hm\!=\!m^{15},{}^jm\!=\!m^{112},{}^km\!=\!m^{112}\rangle$}
Here there are no subgroups in $\cH$ mapping onto $A$,
but the images of two elements of $\cH$ may generate the whole of $A$.
This cannot happen when $C$ has prime order.

\end{remark}

\begin{remark}
\label{remcanrel}
Although there is no a priori preferred representative of any class in
$\Prim(G)$, the generators of $\Prim(G)$ in Theorem \ref{thm:A}
for quasi-elementary $G$ are fairly canonical in the following
sense. The results of \S\ref{sec:qegen} show that in
case 4(c) every primitive $G$-relation
$\Theta = \sum_H n_HH$ satisfies the following conditions:
\begin{itemize}
\item There exists at least two subgroups $H$ of $P$ of maximal size among
those that intersect $\centralCp$ trivially such that 
$n_H\not\equiv 0\mod p$.
\item The sum of $n_H$ over all such $H$ is 0 mod $p$.
\item For any $\tilde{C}\le C$, there exists a subgroup $H$ of $G$
that intersects $\tilde{C}$ non-trivially and such that $n_H\neq 0$.
\end{itemize}
Similar remarks apply to the cases 4(a) and 4(b).
\end{remark}

\section{Examples}

\begin{example}\label{ex:SL2(F3)}
Let $G=\SL_2(\bF_3)$. Its Sylow 2-subgroup $S$ is normal in $G$ and is
isomorphic to the quaternion group $Q_8$. The Sylow 2-subgroup and $G$ itself
are the only non-cyclic subgroups of $G$, so $K(G)$ has rank~2. Since $G$ is
not in the list of Theorem \ref{thm:A}, all its relations come from proper
subquotients. By Theorem \ref{thm:pgroups}, $K(S)$ is generated by the
relation lifted from $S/Z(S)\cong C_2\times C_2$. The only other subquotient
of $G$ that has primitive relations is $G/Z(G)\cong A_4$, which is of type
3(a) in Theorem \ref{thm:A}
with $Q$ cyclic. Combining everything we have said and noting that the three cyclic subgroups
of order 4 in $S$ are conjugate in $G$, we see that $K(G)$ is generated~by
\begin{eqnarray*}
  \Theta_1 & = & C_4-C_6-S+G,\\
  \Theta_2 & = & C_2-3C_4+2S.
\end{eqnarray*}
\end{example}

\begin{example}
Let $G=A_5$. Since $G$ is simple, Theorem
\ref{thm:notquasiel} shows that $G$ has a primitive relation and
$\Prim(G)\cong \bZ$ and is generated by any relation in which $G$
enters with coefficient 1.
Using \cite[Theorem 2.16(i)]{Sna-88} or explicitly 
decomposing all permutation characters in $A_5$ into
irreducible characters, we find that
$$
  \Theta=C_2 - C_3 - V_4 + S_3 - D_{10} + G
$$
is a relation (of the form predicted by Theorem \ref{thm:Solomon}).
Theorem \ref{thm:notquasiel} now implies that all Brauer relations in
$G$ can be expressed as integral linear combinations of $\Theta$ and
of relations coming from proper subgroups.
The non-cyclic proper subgroups of $G$ are $V_4$, $S_3$, $D_{10}$ and $A_4$
and their relations induced to $G$ together with $\Theta$ generate $K(G)$.
\end{example}

\begin{example}\label{ex:wreath}
Let $G=C_3\wr C_4$ be the wreath product of $C_3$ by $C_4$. Then the
subspace of $\bF_3^4$ on which $C_4$ acts trivially is a normal subgroup of
$G$ with non-quasi-elementary quotient. Thus, all relations of $G$ are
obtained from proper subquotients by Corollary \ref{cor:quasielquotients}.
\end{example}

\section{An application to regulator constants}
\label{s:appl}

Let $\Theta=\sum_H n_H H$ be a Brauer relation in a group $G$. Write
\[
  \cC_\Theta(\triv) = \prod_H |H|^{-n_H}.
\]
This quantity is called the {\em regulator constant} of the trivial 
$\Z G$-module. We refer the reader to \cite{tamroot} \S2.2 and \cite{Bar-09} \S2.2
for the definition of regulator constants for general $\Z G$-modules
and their properties.
%
Note that $\cC_\Theta(\triv)$ is invariant under induction of
$\Theta$ from subgroups and lifts from quotients 
(using $\sum n_H=\langle\Theta,\triv\rangle=0$), and that 
$\cC_{\Theta+\Theta'}(\triv)=\cC_\Theta(\triv)\cC_{\Theta'}(\triv)$.

As an application of Theorem \ref{thm:A}, we classify,
given a prime number $l$, all finite groups $G$ that have a Brauer
relation $\Theta$ with the property that $\ord_l(\cC_\Theta(\triv))\neq 0$.
Here, $\ord_l$ denotes the (additive) $l$-adic valuation of a non-zero rational
number.
For an example of number theoretic consequences of the theorem, see~\cite{Bar-12}. 
\begin{theorem}\label{thm:regconsttriv}
  Let $G$ be a finite group and $l$ a prime number.
  Then any Brauer relation in $G$ is a sum of a relation
  $\Theta'$ satisfying $\ord_l\cC_{\Theta'}(\triv)=0$ and relations 
  induced and/or lifted from relations $\Theta$ from 
  subquotients $H=(C_l)^d\rtimes Q$ of the following form:
  \begin{enumerate}
    \item
      $d=1$, $Q=C_{p^{k+1}}$, $p\neq l$ prime, $Q$ acting faithfully on $C_l$;
    $\Theta=C_{p^k} - pQ-C_l\rtimes C_{p^k}+pH$; $\cC_\Theta(\triv) = l^{-p+1}$.
    \item
      $d=1$, $Q=C_{mn}$ acting fathfully on $C_l$, $(m,n)=1$, $m\alpha+n\beta=1$;
    $\Theta = H - Q +\alpha(C_n-C_l\rtimes C_n) + \beta(C_m-C_l\rtimes C_m)$;
    $\cC_\Theta(\triv) = l^{\alpha+\beta-1}$.
    \item
      $d>1$, either $Q$ is quasi-elementary and acts faithfully irreducibly
      on $(C_l)^d$, or $H=(C_l\rtimes P_1)\times (C_l\rtimes P_2)$, where
      $P_1$, $P_2$ are cyclic $p$-groups, $p\neq l$, acting faithfully on the
      respective $C_l$;
      \[
      \Theta = H-Q +
      \sum_{U\in\cU}
      \left(U\rtimes N_QU - (C_l)^d\rtimes N_QU\right);
      \]
      $\cC_\Theta(\triv) = l^{|\cU|-d}$,
      where $\cU$ is the set of index $l$ subgroups of $(C_l)^d$ up to 
      $Q$-conjugation.
  \end{enumerate}
\end{theorem}

\begin{corollary}
\label{cor:regconstsubquo}
  A group $G$ has a Brauer relation $\Theta$ with
  $\ord_l(\cC_\Theta(\triv))\neq 0$ if and only if it has a subquotient
  isomorphic either to $C_l\times C_l$ or to $C_l\rtimes C_p$ with $C_p$ of
  prime order acting faithully on $C_l$.
\end{corollary}
\begin{proof}
If $G$ has a subquotient $C_l\times C_l$, respectively $C_l\rtimes C_p$, then
the induction/lift of a relation from Example \ref{ex:cpcp}, respectively 
\ref{exer:G20} is as required. The
converse immediately follows from Theorem \ref{thm:regconsttriv}, noting that
all groups listed there have a subquotient of the required type.
\end{proof}

We begin by reducing the theorem to soluble groups.
\begin{definition}
Given a prime number $l$, write $\Z_l$ for the ring of $l$-adic integers.
We call a Brauer relation 
$
  \Theta = \sum_H H - \sum_{H'}H'
$ in $G$
a \emph{$\Z_l$-isomorphism} if 
$$
  \bigoplus_H \Z_l[G/H] \iso \bigoplus_{H'} \Z_l[G/H'],
$$
or equivalently (see \cite[Lemma 5.5.2]{Benson}) if
\[
\bigoplus_H \F_l[G/H] \iso \bigoplus_{H'} \F_l[G/H'].
\]
\end{definition}
The following result is a slight strengthening of \cite[Theorem 5.6.11]{Benson}:
\begin{theorem}[Dress's induction theorem]\label{thm:Dress}
Let $G$ be a finite group and $l$ a prime number. There exists a $\Z_l$-isomorphism
in $G$ of the form
$$
  G + \sum_H \alpha_H H,\quad\qquad \alpha_H\in \Z,
$$
the sum taken over those subgroups $H$ of $G$
for which $H/O_l(H)$ is quasi-elementary. Here $O_l(H)$ is 
the $l$-core of $H$ (the largest normal $l$-subgroup).
\end{theorem}
\begin{proof}[Sketch of the proof]
This is shown in the course of the proof of \cite[Theorem 5.6.11]{Benson},
but since the actual statement of the theorem is somewhat weaker,
we summarise for the benefit of the reader the main ideas of the proof.
It is enough to prove that for any prime number $q$, there exists a
$\Z_l$-isomorphism in $G$ of the form
\[
aG + \sum_H \alpha_H H,
\]
where the sum is over subgroups $H$ for which $H/O_l(H)$ is quasi-elementary,
$\alpha_H\in \Z$, and $a\in \Z$ is not divisible by $q$. In other words,
it is enough to exhibit suitable elements of $B(G)\otimes \Z_{(q)}$ that become
trivial under the natural map $B(G)\otimes \Z_{(q)}\rightarrow R_{\F_l}(G)\otimes \Z_{(q)}$.
To do that, one first writes $\triv\in B(G)\otimes \Z_{(q)}$ as a sum of primitive
idempotents $\triv = \sum_H e_H$, which are described in
\cite[Corollary 5.4.8]{Benson}, with the property that each $e_H$ is induced from
$B(H)\otimes \Z_{(q)}$ (\cite[Theorem 5.4.10]{Benson}).
One then shows that under the map
\[
B(G)\otimes \Z_{(q)}\longrightarrow R_{\F_l}\otimes \Z_{(q)},
\]
only those $e_H$ map to non-zero idempotents, for which $H/O_l(H)$ is $q$-quasi-elementary.
Since each $e_H$ is a linear combination of $G$-sets $G/U$, $U\le H$, with
coefficients whose denominators are not divisible by $q$, the result follows.
\end{proof}

\begin{corollary}\label{cor:Dress}
  Let $G$ be a finite group and $l$ a prime number. Any Brauer relation can be written
  as a sum of relations induced from soluble subgroups of $G$ and a $\Z_l$-isomorphism.
\end{corollary}
\begin{proof}
  Let $\Theta$ be an arbitrary Brauer relation in $G$, let $R=\triv_G +
  \sum_H \alpha_H H$ be a $\Z_l$-isomorphism, as given by Theorem \ref{thm:Dress}.
  In particular, all subgroups $H$ in the sum are soluble. Since the subgroup
  of $B(G)$ that consists of $\Z_l$-isomorphisms forms an ideal in $B(G)$, we
  see that
  \[
  \Theta\cdot R = \Theta + \sum_H \alpha_H \Ind^G\Res_H \Theta
  \]
  is a $\Z_l$-isomorphism, and the claim is established.
\end{proof}

\begin{proof}[Proof of Theorem \ref{thm:regconsttriv}]
It is easy to see that if $R$
is a $\Z_l$-isomorphism, then $\ord_l(\cC_R(\triv))=0$ (and in fact,
the same is true with $\triv$ replaced by any finitely generated $\Z[G]$-module).
Thus, Corollary \ref{cor:Dress} reduces the proof of the theorem to the case
that $G$ is soluble.

Writing $\Theta$ as a sum of primitive relations listed in Theorem
\ref{thm:A}, we see immediately by inspection that the relations
$\Theta'$ that generate $\Prim(G)$ in the cases 1, 2, and 4 satisfy
$\cC_{\Theta'}(\triv) =1$. The remaining assertions of the theorem
follow from a direct calculation for the generators of $\Prim(G)$ in the
case 3.
\end{proof}


\end{document}